\documentclass[a4paper,12pt,leqno]{amsart}
\usepackage{amsmath,amscd,amsthm,amssymb,amsxtra,latexsym,epsfig,epic,graphics}
\usepackage[matrix,arrow,curve]{xy}
\usepackage{graphicx}
\usepackage[breaklinks,bookmarksopen,bookmarksnumbered]{hyperref}
\usepackage[hyphenbreaks]{breakurl}
\hypersetup{colorlinks=true,
	citecolor=blue}
\usepackage[table]{xcolor}

\usepackage[english]{babel}
\usepackage[T1]{fontenc}
\usepackage[disable] 
{todonotes}
\usepackage{mathtools} 
\usepackage{booktabs} 
\usepackage{mathrsfs} 
\usepackage[color]{changebar} 
\cbcolor{black}
\newcommand{\scr}{\mathscr}

\newtheorem*{theorem*}{Theorem}

\newtheorem{question}{Question}

\def\antiddot{\mathinner{\mkern1mu\raise1pt\vbox{\kern7pt\hbox{.}}\mkern2mu
		\raise4pt\hbox{.}\mkern2mu\raise7pt\hbox{.}\mkern1mu}}

\newcommand{\CC}{{\mathbb C}}

\newcommand{\GG}{{\mathbb G}}

\newcommand{\KK}{{\mathbb K}}
\newcommand{\PP}{{\mathbb P}}
\newcommand{\QQ}{{\mathbb Q}}

\newcommand{\ZZ}{{\mathbb Z}}

\newcommand{\LL}{{\mathbb L}}
\newcommand{\HH}{{\rm{H}}}

\newcommand{\id}{{\rm{id}}}

\newcommand{\coker}{{\rm{coker}\,}}

\newcommand{\s}{\mathcal}

\newcommand{\sD}{{\s D}}

\newcommand{\sF}{{\s F}}
\newcommand{\sG}{{\s G}}
\newcommand{\sH}{{\s H}}
\newcommand{\sI}{{\s I}}
\newcommand{\sJ}{{\s J}}
\newcommand{\sK}{{\s K}}
\newcommand{\sL}{{\s L}}
\newcommand{\sM}{{\s M}}
\newcommand{\sN}{{\s N}}
\newcommand{\sO}{{\s O}}

\newcommand{\sT}{{\s T}}

\newcommand{\sV}{{\s V}}
\newcommand{\sW}{{\s W}}

\newcommand{\tensor}{\otimes}

\newcommand{\punkt}{\hspace{-.3ex}\raise.15ex\hbox to1ex{\Huge.}}

\DeclareMathOperator{\Pic}{Pic}

\DeclareMathOperator{\Hilb}{Hilb}

\DeclareMathOperator{\Spec}{Spec}

\DeclareMathOperator{\sHom}{\sH om}

\DeclareMathOperator{\supp}{supp}

\DeclareMathOperator{\image}{image}

\DeclareMathOperator{\pd}{pd}

\DeclareMathOperator{\PGL}{PGL}
\DeclareMathOperator{\Tor}{Tor}

\DeclareMathOperator{\codim}{codim}
\DeclareMathOperator{\rank}{rank}

\newcommand{\Mac}{{\texttt {Macaulay2 }}}
\newtheorem{theorem}{Theorem}[section]
\newtheorem{lemma}[theorem]{Lemma}
\newtheorem{proposition}[theorem]{Proposition}
\newtheorem{corollary}[theorem]{Corollary}

\theoremstyle{definition}
\newtheorem{definition}[theorem]{Definition}

\newtheorem{remark}[theorem]{Remark}

\newtheorem{example}[theorem]{Example}

\newtheorem{algorithm}[theorem]{Algorithm}

\theoremstyle{definition}
\newtheorem*{notation*}{Notation}
\newtheorem*{ack}{Acknowledgements}
\numberwithin{table}{section}

\newcommand{\hh}{{\rm{h}}}

\newcommand\rcell{\cellcolor{red!10}}
\newcommand\bcell{\cellcolor{blue!10}}

\newcommand{\bettit}[1]{
	\begin{array}{c|cccc}
		& 	\rule{0.5ex}{0pt} 0 \rule{0.5ex}{0pt} &	\rule{0.5ex}{0pt} 1 \rule{0.5ex}{0pt} & \rule{0.5ex}{0pt} 2 \rule{0.5ex}{0pt} &		\rule{0.5ex}{0pt} 3 \rule{0.5ex}{0pt}\\ \hline
		#1
	\end{array}}
	
\newcommand{\bettif}[1]{
	\begin{array}{c|ccccc}
		& 	\rule{0.5ex}{0pt} 0 \rule{0.5ex}{0pt} &	\rule{0.5ex}{0pt} 1 \rule{0.5ex}{0pt} & \rule{0.5ex}{0pt} 2 \rule{0.5ex}{0pt} &		\rule{0.5ex}{0pt} 3 \rule{0.5ex}{0pt} & \rule{0.5ex}{0pt} 4 \rule{0.5ex}{0pt}  \\ \hline
		#1
	\end{array}}

\DeclareMathOperator{\Vi}{V}
\let\image\relax 
\DeclareMathOperator{\image}{im}

\makeatletter
\def\Ddots{\mathinner{\mkern1mu\raise\p@
		\vbox{\kern7\p@\hbox{.}}\mkern2mu
		\raise4\p@\hbox{.}\mkern2mu\raise7\p@\hbox{.}\mkern1mu}}
\makeatother

\begin{document}
	
	\title{Matrix factorizations and curves in $\PP^4$}
	\author{Frank-Olaf Schreyer}
	\address{Mathematik und Informatik\\
		Universit\"at des Saarlandes\\
		Campus E2.4\\
		D-66123 Saarbr\"ucken\\
		Germany}
	\email{schreyer@math.uni-sb.de}
	\author{Fabio Tanturri}
	\address{Institut de Math\'ematiques de Marseille\\
		Aix-Marseille Universit\'e\\
		Technop\^ole Ch\^ateau-Gombert\\
		13453 Marseille Cedex 13\\
		France }
	\email{fabio.tanturri@univ-amu.fr}
	
	\subjclass[2010]{}
	
	
	\date{\today}
	
	\begin{abstract}
	Let $C$ be a curve in $\PP^4$ and $X$ be a hypersurface containing it. We show how it is possible to construct a matrix factorization on $X$ from the pair $(C,X)$ and, conversely, how a matrix factorization on $X$ leads to curves lying on $X$. We use this correspondence to prove the unirationality of the Hurwitz space $\sH_{12,8}$ and the uniruledness of the Brill--Noether space $\sW^1_{13,9}$. Several unirational families of curves of genus $16 \leq g \leq 20$ in $\PP^4$ are also exhibited. 
	\end{abstract}
	
	\maketitle
	
	\section*{Introduction}
	
	The moduli space $\sM_g$ of curves of genus $g$ is known to be unirational for $g \leq 14$ \cite{Severi, SernesiUnirationality, ChangRanUnirationality, VerraUnirationality}, while for $g=22$ or $g \geq 24$ it is proved to be of general type \cite{HarrisMumfordKodaira, EisenbudHarrisKodaira, FarkasGeometry, FarkasBirational}. For the cases in between, only partial results are available: $\sM_{23}$ has positive Kodaira dimension \cite{FarkasGeometry}, $\sM_{15}$ is rationally connected \cite{ChangRanKodaira, BrunoVerraRationally} and $\sM_{16}$ is uniruled \cite{ChangRanSlope, FarkasBirational}.
	
	Similarly, the unirationality  of  Hurwitz spaces $\sH_{g,d}$ parameterizing $d$-sheeted branched simple covers of the projective line by smooth curves of genus $g$ is of fundamental interest. For small values of $d$ or $g$ they are proven to be unirational, but for larger values few results are known. See Section \ref{unirationalHurwitzSpaces} for a  discussion on the known results.

	In this paper we introduce a correspondence between (general) curves $C$ in $\PP^4$ with fixed genus and degree, together with a hypersurface $X\supset C$, and the space of certain matrix factorizations on $X$. This leads to a new technique to construct curves in $\PP^4$, which has been positively used by Schreyer \cite{SchreyerMatrix} in the particular case of curves of genus $15$ and degree $16$.
	
	The goal of this paper, in addition to showing how matrix factorizations can be used to construct curves in $\PP^4$, is to use this technique to  prove new positive results. Our main contribution is the following
	
	\begin{theorem*}[Theorem \ref{unirationalityThm}]
	$\sH_{12,8}$ is unirational.
	\end{theorem*}
	
	To prove this result, we construct explicitly a unirational dominant family of curves of genus 12 and degree 14 in $\PP^4$ by means of matrix factorizations, showing thus that the Brill--Noether space $\sW^4_{12,14}$ is unirational. A general point $(C,L)$ in $\sW^4_{12,14}$ gives rise to a point $(C,K_C-L)$ in $\sW^1_{12,8}$ and conversely, whence the unirationality of $\sW^1_{12,8}$ and $\sH_{12,8}$. 
	
	The study of the correspondence between curves and matrix factorizations in another particular case leads to a very cheap proof of the following
		
	\begin{theorem*}[Corollary \ref{unirulednessThm}]
	$\sW^1_{13,9}$ is uniruled.
	\end{theorem*}
	
	The same method yields a proof of the uniruledness of $\sW^1_{12,8}$, already implied by the previous theorem, and of $\sW^1_{11,7}$ and $\sW^1_{10,6}$, already known to be unirational \cite{GeissUnirationality,GeissThesis}. 
	
	In Section \ref{unirationalHurwitzSpaces} we will formulate some speculations and questions about the range of unirational Hurwitz spaces, which partly motivates our study; we remark that the unirationality of $\sH_{12,8}$ and the uniruledness of $\sW^1_{13,9}$ fit perfectly into the picture.
	
	Matrix factorizations can be used constructively more in general. We present a way to construct unirational families of curves of genus $g \in [16,20]$; even though these families will be far from being dominant on $\sM_g$, such concrete examples offer the chance to prove some other results. For instance, we are able to prove the following
	
	\begin{theorem*}[Theorem \ref{generalQuartic}]
		A general cubic hypersurface in $\PP^4$ contains a family of dimension $2d$ of curves of genus $g$ and degree $d$ for
		\[
		(g,d) \in \{(12,14), (13,15)\}.
		\]
		A general quartic hypersurface in $\PP^4$ contains a $d$-dimensional family of curves of genus $g$ and degree $d$ for
		\[
		(g,d) \in \{(16,17), (17,18), (18,19), (19,20), (20,20)\}.
		\]
		\end{theorem*}
	
The construction of our families of curves of genus $g \in [16,20]$ relies on considering particular rational surfaces arising when trying to adapt our technique to these specific cases. Other instances of results which can be proved by looking at specific examples concern the structure of the syzygies of general curves of particular genera and degrees, as mentioned in Theorem \ref{constructionThm}.

In the paper, we will often need to exhibit a concrete example to prove that some open conditions are generally satisfied. Our explicit constructions are performed by means of the software \Mac \cite{M2} and run best over a finite field. Semicontinuity arguments will ensure the existence of suitable examples over the rational or the complex field as well, as explained in Remark \ref{posCharSuffices}. For the supporting documentation regarding the computational proofs contained in this paper, we will always refer to \cite{SchreyerTanturriCode}. \newline

The paper is structured as follows: in Section \ref{unirationalHurwitzSpaces} we survey the known results about the unirationality of Hurwitz spaces and we present some questions and speculations about what kind of general behavior can be expected. In Section \ref{matrixFact} we recall some basic definitions and general facts about matrix factorizations and we explain, starting with a motivating example, the correspondence between particular matrix factorizations and curves in $\PP^4$. 
The key point of the correspondence is the Reconstruction Theorem \ref{reconstructionThm}. In Section \ref{uniruledness} we prove Theorem \ref{constructionThm}, which gives us an effective method to produce curves in $\PP^4$ starting from suitable matrix factorizations; moreover, we use the previous correspondence to provide a cheap proof of the uniruledness of $\sW^1_{13,9}$ (Corollary \ref{unirulednessThm}). In Section \ref{uniratHurwitz} we prove our main result, Theorem \ref{unirationalityThm}; for this sake, we use particular matrix factorizations arising from suitable auxiliary curves of genus 10 and degree 13. Finally, in Section \ref{families} we construct unirational families of curves of genus $16 \leq g \leq 20$ lying on particular rational surfaces in $\PP^4$.

\begin{ack}
	The authors would like to thank the referee for valuable suggestions and remarks.
\end{ack}

\begin{notation*}
In the paper we will use \Mac notation for Betti tables. If a module $M$ has Betti numbers $\beta_{i,j}= \dim \Tor_i^R(M,{K})_j $ over a ring $R$ with base field ${K}$, its Betti table will be written as  \rule{0pt}{0pt}
\[
\begin{array}{c|ccccc}
& 	\rule{0.5ex}{0pt} 0 \rule{0.5ex}{0pt} &	\rule{0.5ex}{0pt} 1 \rule{0.5ex}{0pt} & \rule{0.5ex}{0pt} 2 \rule{0.5ex}{0pt} & \dotso \\ \hline
0& \beta_{0,0} & \beta_{1,1} &	\beta_{2,2} &\dotso \\
1&	\beta_{0,1} & \beta_{1,2} & \beta_{2,3}	& \dotso\\
2& \beta_{0,2} & \beta_{1,3} & \beta_{2,4} & \dotso\\
\vdots& \vdots & \vdots &\vdots & \ddots
\end{array}
\] 
\end{notation*}


\section{Unirationality of Hurwitz spaces}\label{unirationalHurwitzSpaces}

In this section we briefly survey what we know about the unirationality of the Hurwitz spaces $\sH_{g,d}$. To put the question into the right framework we recall a few facts from Brill--Noether theory.
 
 A general curve $C$ of genus $g$ has a linear system $g^r_d$ of dimension $r$ of
divisors of degree $d$ if and only if  the Brill--Noether number
$$ \rho=\rho(g,r,d)=g-(r+1)(g+r-d)$$
is non-negative.
Moreover, in this case, the Brill--Noether scheme
$$W^r_d(C)=\{ L \in \Pic^d(C) \mid \hh^0(L) \ge r+1 \}$$
 has dimension $\rho$. 
Recall some notation from \cite{ACGH}:
$$
\sM^r_{g,d} = \{ C \in \sM_g \mid \exists L \in W^r_d(C) \}, \;
$$
$$
\sW^r_{g,d} = \{ (C,L) \mid C \in \sM^r_{g,d}, L \in W^r_d(C) \},
$$
$$
\sG^r_{g,d} = \{ (C,L,V) \mid (C,L) \in \sW^r_{g,d}, V \subset \HH^0(L), \dim V =r+1 \}.
$$
Thus we have natural morphisms
$$
\xymatrix{ 
\sH_{g,d} \ar[r]^\alpha & \sG^1_{g,d} \ar[r]^\beta & \sW^1_{g,d} \ar[r]^\gamma & \sM^1_{g,d}; \cr}
$$
with our notation, $\alpha$ is a $\PGL(2)$-bundle over the base point free locus, with fibers corresponding to the choices of a basis of $V$, the fibers of $\beta$ are Grassmannians $\GG(2,\HH^0(C,L))$, and the fibers of $\gamma$ are the $W^1_d(C)$. Thus the unirationality of $\sH_{g,d}$ is equivalent to the unirationality of $\sW^1_{g,d}$.

The unirationality of $\sH_{g,d}$ for $2 \le d \le 5$ and arbitrary $g\ge 2$ has been known for a long time. The case $d=5$ is due to Petri \cite{Petri}, with clarification given by the Buchsbaum--Eisenbud structure Theorem \cite{BuchsbaumEisenbud,SchreyerSyzygies}, and independently to B.~Segre \cite{Segre}, with clarification by Arbarello and Cornalba \cite{ArbarelloCornalba}.

The case for $g \le 9$ is due to Mukai:

\begin{theorem}[Mukai \cite{Mukai}]
A general canonical curve $C$ of genus $g=7,8,9$ arises as transversal intersection of a linear space with a homogeneous variety: 
 \rule{0pt}{0pt}

\begin{center}
\begin{tabular}{ccc}
\midrule
\rule{5pt}{0pt}$7$\rule{5pt}{0pt} & \rule{5pt}{0pt}$C = \PP^{6} \cap {\rm Spinor}^{10} \subset \PP^{15}$\rule{5pt}{0pt} & \rule{5pt}{0pt}isotropic subspaces of $Q^8 \subset \PP^9$\rule{5pt}{0pt} \\
$8$ & $C = \PP^{7} \cap \GG(2,6)^8 \subset \PP^{14}$ & Grassmannian of lines in $\PP^5$ \\
$9$ & $C = \PP^{8} \cap \LL(3,6)^6 \subset \PP^{13}$ & Lagrangian subspaces of $(\CC^6,\omega)$ \\
\midrule
\end{tabular}
\end{center}

\end{theorem}

Structure results for canonical curves of genus $g \le 6$ are classical, see, e.g., \cite{SchreyerSyzygies}.

\begin{corollary}
 The moduli spaces $\sM_{g,g}$ of $g$-pointed curves of genus $g$ and the universal Picard varieties $\Pic^d_g$  are unirational for $g\le 9$ and any $d$. The spaces
$ \sM^1_{g,d}$ and $\sH_{g,d}$ are unirational for $g \le 9$ and $d \ge g$.
\end{corollary}

\begin{proof} The argument is the same as in \cite[\textsection 1]{VerraUnirationality}. We can choose $g$ general points $p_1,\ldots,p_g$ in the homogeneous variety and can take $\PP^{g-1}$ as their span. Then the intersection of the homogeneous variety with this $\PP^{g-1}$ gives a smooth curve $C$ of genus $g$
together with $g$ marked points. For the line bundle, we may take $L= \sO_C(\sum_{j=1}^g d_j p_j)$ for integers $d_1, \ldots, d_g$ with $\sum_{j=1}^g d_j =d$.

As for the unirationality of $\sM^1_{g,d}$ for $d\ge g+1$, with $L$ as above we have $\hh^0(C,L) \ge 2$. In case $d=g$, we take $L=\omega_C(-\sum_{j=1}^{g-2} p_j)$, which is a line bundle  $L \in W^1_g(C) \setminus W^2_g(C)$ by 
Riemann--Roch
. The unirationality of $\sH_{g,d}$ then follows. \qedhere
\end{proof}

In the range $d \le 5$ or $g\le 9$, apart from a few cases due to Florian Gei\ss\phantom{ }\cite{GeissThesis}, only the unirationality of $\sH_{9,8}$ needed to be proved. This has recently been established in \cite{DamadiSchreyer}.

\begin{figure}
\begin{scriptsize}
\begin{tabular}{|c|cccc|ccccccccc}
{45}&\bcell&\bcell&\bcell&\bcell{\color{blue}P} &\bcell{\color{blue}G}&\rcell&\rcell&\rcell&\rcell&\rcell&\rcell&\rcell&\rcell \cr
$\mid$&\bcell&\bcell&\bcell& \bcell{\color{blue}$\mid$} &&\rcell&\rcell&\rcell&\rcell&\rcell&\rcell&\rcell&\rcell\\
$\mid$&\bcell&\bcell&\bcell& \bcell{\color{blue}$\mid$} &&\rcell&\rcell&\rcell&\rcell&\rcell&\rcell&\rcell&\rcell\\
{40}&\bcell&\bcell&\bcell&\bcell{\color{blue}P} &\bcell{\color{blue}G}&\rcell&\rcell&\rcell&\rcell&\rcell&\rcell&\rcell&\rcell \\
$\mid$&\bcell&\bcell&\bcell& \bcell{\color{blue}$\mid$} &&\rcell&\rcell&\rcell&\rcell&\rcell&\rcell&\rcell&\rcell\\
$\mid$&\bcell&\bcell&\bcell& \bcell{\color{blue}$\mid$} &&\rcell&\rcell&\rcell&\rcell&\rcell&\rcell&\rcell&\rcell\\
{36}&\bcell&\bcell&\bcell&\bcell{\color{blue}P} &\bcell{\color{blue}G}&\rcell&\rcell&\rcell&\rcell&\rcell&\rcell&\rcell&\rcell \\
{35}&\bcell&\bcell&\bcell&\bcell{\color{blue}P} &\bcell{\color{blue}G}&\rcell&\rcell&\rcell&\rcell&\rcell&\rcell&\rcell&\rcell\\
34&\bcell&\bcell&\bcell& \bcell{\color{blue}P} &&\rcell&\rcell&\rcell&\rcell&\rcell&\rcell&\rcell&\rcell\\
{33}&\bcell&\bcell&\bcell&\bcell{\color{blue}P} &\bcell{\color{blue}G}&\rcell&\rcell&\rcell&\rcell&\rcell&\rcell&\rcell&\rcell\\
32&\bcell&\bcell&\bcell&\bcell  {\color{blue}P}&&\rcell&\rcell&\rcell&\rcell&\rcell&\rcell&\rcell &\rcell \\
{31}&\bcell&\bcell&\bcell&\bcell {\color{blue}P}&\bcell{\color{blue}G}&\rcell &\rcell&\rcell&\rcell&\rcell&\rcell&\rcell&\rcell\\
{30}&\bcell&\bcell&\bcell&\bcell{\color{blue}P} &\bcell{\color{blue}G}&\rcell&\rcell&\rcell&\rcell&\rcell&\rcell&\rcell&\rcell\\
29&\bcell&\bcell&\bcell&\bcell  {\color{blue}P}&&\rcell&\rcell&\rcell&\rcell&\rcell&\rcell&\rcell &\rcell \\
{28}&\bcell&\bcell&\bcell&\bcell{\color{blue}P} &\bcell{\color{blue}G}&\rcell&\rcell&\rcell&\rcell&\rcell&\rcell&\rcell&\rcell\\
27 &\bcell&\bcell&\bcell&\bcell {\color{blue}P}&\bcell{\color{blue}G}&&\rcell&\rcell&\rcell&\rcell&\rcell&\rcell&\rcell\\
{26}&\bcell&\bcell&\bcell&\bcell{\color{blue}P} &\bcell{\color{blue}G} &&\rcell&\rcell&\rcell&\rcell&\rcell&\rcell&\rcell{\color{red}$  {\hbox{EH}}$} \\
{25}&\bcell&\bcell&\bcell&\bcell{\color{blue}P} &\bcell{\color{blue}G} &&\rcell&\rcell&\rcell&\rcell&\rcell&\rcell&\rcell{\color{red}$  {\hbox{HM}}$} \\
{24}&\bcell&\bcell&\bcell&\bcell{\color{blue}P} &\bcell{\color{blue}G} &&\rcell&\rcell&\rcell&\rcell&\rcell&\rcell{\color{red}$  {\hbox{EH}}$}&\rcell{\color{red}$  {\hbox{EH}}$} \\
{23}&\bcell&\bcell&\bcell&\bcell{\color{blue}P} &\bcell{\color{blue}G}&&\rcell&\rcell&\rcell&\rcell&\rcell&\rcell{\color{red}HM}&\rcell{\color{red}HM}\\
{22}&\bcell&\bcell&\bcell&\bcell {\color{blue}P}&\bcell {\color{blue}G}&&\rcell&\rcell&\rcell&\rcell&\rcell{\color{red}$  {\;\hbox{F}\;}$}  &\rcell{\color{red}$  {\;\hbox{F}\;}$}&\rcell{\color{red}$  {\;\hbox{F}\;}$}\\ \hline
21&\bcell&\bcell&\bcell&\bcell {\color{blue}P}&\bcell {\color{blue}G}&&\rcell&\rcell&\rcell&\rcell&\rcell{\color{black}} &\rcell&\rcell\\
$\mid$ &\bcell&\bcell&\bcell&\bcell{\color{blue}$\mid$}&\bcell{\color{blue}$\mid$}&&\rcell&\rcell&\rcell&\rcell&\rcell&\rcell&\rcell\\
$\mid$ &\bcell&\bcell&\bcell&\bcell{\color{blue}$\mid$}&\bcell{\color{blue}$\mid$}&&\rcell&\rcell&\rcell&\rcell&\rcell&\rcell&\rcell\\
$\mid$ &\bcell&\bcell&\bcell&\bcell{\color{blue}$\mid$}&\bcell{\color{blue}$\mid$}&&\rcell&\rcell&\rcell&\rcell&\rcell&\rcell&\rcell\\
16&\bcell&\bcell&\bcell&\bcell{\color{blue}P}&\bcell{\color{blue}G}&&\rcell&\rcell$  { \; \;}$&\rcell&\rcell&\rcell&\rcell&\rcell\\
15&\bcell&\bcell&\bcell&\bcell{\color{blue}P}&\bcell{\color{blue}G}&&\rcell&\rcell{\color{violet}V}&\rcell&\rcell&\rcell&\rcell&\rcell\\
14&\bcell&\bcell&\bcell&\bcell{\color{blue}P}&\bcell{\color{blue}G}  & &\bcell{\color{blue}$  {\; \hbox{V} \;}$} &\rcell&\rcell&\rcell&\rcell&\rcell&\rcell{\color{red} FV}\\
13&\bcell&\bcell&\bcell&\bcell{\color{blue}P}&{\bcell\color{blue}G}&\bcell {\color{blue}KT}&&\rcell{\color{violet}ST}&\rcell&\rcell&\rcell  &\rcell{\color{red} FV}&\rcell{\color{red} CKV}\\
12&\bcell&\bcell&\bcell&\bcell{\color{blue}P}&\bcell{\color{blue}G}&\bcell{\color{blue}$  { \,\hbox{G} \,}$}&\bcell{\color{blue}ST}&\bcell{\color{blue}S}&\rcell&\rcell&\rcell{\color{red} FV}&\rcell{\color{red} CKV}  &\rcell{\color{red} CKV}\\
11&\bcell&\bcell&\bcell&\bcell{\color{blue}P}&\bcell{\color{blue}G}&\bcell{\color{blue}G}&\bcell{\color{blue}CR}&&\rcell&\rcell{\color{violet} FV}&\rcell{\color{red} CKV}&\rcell  {\color{red} CKV}&\rcell{\color{red} CKV}\\
10&\bcell&\bcell&\bcell&\bcell{\color{blue}P}&\bcell{\color{blue}G}&\bcell{\color{blue}G}&\bcell{\color{blue}KT}&&{\color{violet} FV}&\rcell{\color{red} CKV} &\rcell{\color{red} CKV}  &\rcell{\color{red} CKV}&\rcell{\color{red} BFV}\\
\hline
{\bf \color{blue}9}&\bcell&\bcell&\bcell&\bcell{\color{blue}P}&\bcell{\color{blue}G} &\bcell{\color{blue}G}&\bcell{\color{blue}DS}&\bcell{\color{blue}M}&\bcell
{\color{blue}M}&\bcell{\color{blue}M}&\bcell{\color{blue}M}&\bcell{\color{blue}M}&\bcell{\color{blue}M} \\
{\bf \color{blue}8}&\bcell&\bcell&\bcell&\bcell{\color{blue}P}&\bcell{\color{blue}$\mid$} &\bcell {\color{blue}G}&\bcell{\color{blue}M}&\bcell{\color{blue}M}&\bcell{\color{blue}M}&\bcell{\color{blue}M}&\bcell {\color{blue}M}&{\bcell\color{blue}M}&\bcell{\color{blue}M}\\
{\bf \color{blue}7}&\bcell &\bcell&\bcell &\bcell{\color{blue}P}&\bcell{\color{blue}$\mid$}  &\bcell{\color{blue}M}&
\bcell{\color{blue}M} & \bcell{\color{blue}M}  &\bcell{\color{blue}M}& \bcell{\color{blue}M}  &\bcell{\color{blue}M}&\bcell{\color{blue}M}  &\bcell{\color{blue}M} \\ \hline
{\bf \color{blue}6}&\bcell {\color{blue}}& \bcell {\color{blue}}&\bcell{\color{blue}} &\bcell{\color{blue}$\,$}&\bcell{\color{blue}$\,$}  &\bcell
{\color{blue}}&\bcell  &\bcell&\bcell&\bcell&\bcell&\bcell&\bcell \\
{\bf \color{blue}$\mid$}&\bcell {\color{blue}$\;$}&\bcell &\bcell&\bcell&\bcell&\bcell&\bcell&\bcell&\bcell&\bcell&\bcell&\bcell  &\bcell   \\
{\bf \color{blue}1}&\bcell {\color{blue}$\;$}&\bcell &\bcell&\bcell&\bcell&\bcell&\bcell&\bcell&\bcell&\bcell&\bcell&\bcell &\bcell    \\  \hline  
$ g\;\slash\; d$& {\bf \color{blue}2}& {\bf \color{blue}3} & {\bf \color{blue}4} & {\bf \color{blue}5} & 6 & 7 & 8 & 9 & 10 & 11 & 12 & 13 & 14  \\ \hline
\end{tabular}
\end{scriptsize}

\caption{
 Color coding indicates where $\sW^1_{g, d}$ is known to be  {\color{blue}unirational}, {\color{violet} uniruled}
or {\color{red} not unirational}. \label{Fig1}
Results are due to {\color{blue}M}ukai ($g \le 9$), {\color{blue}P}etri or B.~Segre ($d=5$) \cite{Mukai, Petri, Segre}, 
{\color{red}E}isenbud, {\color{red}H}arris, {\color{red}M}umford, {\color{red}F}arkas, {\color{red}B}ini, {\color{red}C}asalaina-Martin, {\color{red}K}ass, {\color{red}F}ontanari and {\color{red}V}iviani \cite{BiniFontanariViviani, CasalainaKassViviani, EisenbudHarrisKodaira, FarkasGeometry, FarkasBirational, FarkasVerra, HarrisMumfordKodaira}, {\color{blue}C}hang and {\color{blue}R}an, {\color{blue}V}erra, {\color{blue}G}ei\ss , {\color{blue}D}amadi and {\color{blue}S}chreyer, {\color{blue}S}chreyer and  {\color{blue}T}anturri, {\color{blue}K}eneshlou and  {\color{blue}T}anturri \cite{ChangRanUnirationality, ChangRanKodaira, ChangRanSlope, DamadiSchreyer, GeissThesis, GeissUnirationality, KeneshlouTanturri, SchreyerComputer, VerraUnirationality}.
 }
    \end{figure}

Outside the range $d\le 5$ or $g\le 9$ there are only finitely many pairs $(g,d)$ for which $\sH_{g,d}$ is known to be unirational. 

\begin{question} Are there only  finitely many pairs $(g,d)$ with $g\ge 10$ and $d\ge 6$ such that $\sH_{g,d}$
is unirational?  
\end{question}

In particular, we may ask

\begin{question} Are the genera $g$ such that $\sH_{g,6}$ is unirational bounded?
\end{question}

Florian Gei\ss \ \cite{GeissUnirationality} proved the unirationality of $\sH_{g,6}$ for the values $g \in \{9,\ldots,28,30,31,33,35,36,40,45\}$ using models of curves in $\PP^1\times \PP^2$ of bidegree $(6,d_2)$ and liaison, $d_2=d_2(g)$ being the minimal number such that $\rho(g,2,d_2) \ge 0$. His proof actually shows the unirationality of a covering space of $\sW^1_{g,6}$.

\begin{question} Are the genera $g$ such that $\sH_{g,7}$ is unirational bounded?
\end{question}

\begin{question} Is $g=14$ the largest genus  such that $\sH_{g,8}$ is unirational? In other words, is Verra's case \cite{VerraUnirationality}
extremal? Is $g=12$ the largest genus  such that $\sH_{g,9}$ is unirational?
\end{question}

If all these questions have an affirmative answer, then the range of pairs $(g,d)$ such that $\sW^1_{g,d}$ and $\sH_{g,d}$ are not unirational
has roughly shape as indicate in Figure \ref{Fig1} with the color red.

\section{Matrix factorizations and the Reconstruction Theorem}
\label{matrixFact}

\subsection{Matrix factorizations}	
	
Matrix factorizations were introduced by Eisenbud in his seminal paper \cite{EisenbudHomological}. We recall here some basic facts and properties for matrix factorizations over the special case of a polynomial ring $S=K[x_0,\ldots,x_n]$, which is the case of interest for the paper. Any module will be assumed to be finitely generated.

Let $f \in S$ be a nonzero homogeneous form of degree $s$. A \emph{matrix factorization} of $f$ (or on the hypersurface $\Vi(f)$) is a pair $(\varphi, \psi)$ of maps
\[
\varphi: G \to F, \quad \qquad \psi: F \to G(s),
\]
where $F=\bigoplus_{\ell=1}^r S(-a_\ell)$ and $G=\bigoplus_{\ell=1}^{r'} S(-b_\ell)$ are free $S$-modules, satisfying $\psi\circ \varphi = f \cdot \id_G$ and $\varphi(s)\circ \psi = f \cdot \id_F$. This condition forces the two matrices representing the maps to be square, i.e., $r=r'$.

If $(\varphi,\psi)$ is a matrix factorization, then $\coker \varphi$ is a maximal Cohen--Macaulay module (MCM for short) on the hypersurface ring $S/f$. Conversely, a finitely generated MCM $S/f$-module $M$ has a minimal free resolution over $S$
\[
0 \longleftarrow M \longleftarrow F\longleftarrow  G \longleftarrow 0;
\]
multiplication by $f$ on this complex is null homotopic
\[
\xymatrix{ 
	0 & \ar[l] M \ar[d]_0&  \ar[l] F \ar[d]_f \ar@{.>}[dr]^-{\exists\psi}& \ar[l]_\varphi G \ar[d]^f &  \ar[l] 0 \\
	0 & \ar[l] M(s) &  \ar[l] F(s)  & \ar[l]^{\varphi(s)} G(s) &  \ar[l] 0 \\
} 
\]
and yields therefore a matrix factorization $(\varphi,\psi)$. As an $S/f$-module, $M$ has the infinite 2-periodic resolution
\[
\xymatrix{
	0 &\ar[l] M & \ar[l] \overline F&  \ar[l]_{\overline \varphi}  \overline G & \ar[l]_-{\overline{\psi}(-s)} \overline F(-s) &\ar[l]_-{\overline{\varphi}(-s)}  \overline G(-s) & \ar[l]_-{\overline{\psi}(-2s)} \ldots \\
}
\]
where $\overline F=F \tensor S/f$ and $\overline G=G \tensor S/f$. In particular, this sequence is exact, and the dual sequence corresponding to the transposed matrix factorization
$(\psi^t,\varphi^t)$ is exact as well.

If $N$ is an arbitrary $S/f$ module, then any minimal free resolution becomes eventually 2-periodic. If
$$
0 \longleftarrow N \longleftarrow F_0\longleftarrow  F_1 \longleftarrow \ldots \longleftarrow  F_c \longleftarrow 0
$$
is a minimal free resolution of $N$ of length $c$ as an $S$-module, then the Shamash construction \cite{Shamash} produces a (non-necessarily minimal) free resolution of $N$ of the form
\[
0 \leftarrow N \leftarrow \overline F_0\leftarrow  \overline F_1 \leftarrow  
\begin{array}{c}
\overline F_2 \\
\oplus\\
\overline{F_0}(-s)
\end{array}
\leftarrow 
\begin{array}{c}
\overline F_3 \\
\oplus\\
\overline{F_1}(-s)
\end{array}
\leftarrow 
\begin{array}{c}
\overline F_4 \\
\oplus\\
\overline{F_2}(-s)\\
\oplus\\
\overline{F_0}(-2s)
\end{array}
\leftarrow 
\ldots ,
\]
which becomes 2-periodic after the $(c-1)$-th step. This construction allows us to control to some extent the degrees of the entries of the corresponding minimal matrix factorization of $f$ induced by an $S/f$-module $N$, if we know the Betti numbers of $N$ as an $S$-module. The Shamash construction has the following peculiarity: at the $i$-th step
\begin{equation}
\label{shamashConstr}
\xymatrix{\displaystyle\bigoplus_{j\ge 0} \overline F_{i-1-2j}\left(-js \right) & \ar[l] \displaystyle\bigoplus_{j\ge 0} \overline F_{i-2j}(-js)
	}
\end{equation}
the components $\overline F_{i-1-2j}(-js) \leftarrow \overline F_{i-2j}(-js)$ are inherited from the maps $F_{i-1-2j} \leftarrow F_{i-2j}$ in the  resolution of $N$ over $S$ for any $j$, while the component
\begin{equation}
\label{shamashzero}
\xymatrix{\displaystyle\bigoplus_{j\geq 1} \overline F_{i-1-2j}\left(-js \right) & \ar[l] \overline F_{i}
	} \quad \mbox{ is the zero map}.
\end{equation}

\subsection{Curves and matrix factorizations}

An easy way to produce matrix factorizations on a hypersurface $X=\Vi(f)$ in $\PP^4$ is to consider a module $N$ over $S=\KK[x_0,\dotsc,x_4]$ annihilated by $f$. A matrix factorization of $f$ is given by the periodic part of a minimal free resolution of $N$ as a module over $S_X:=S/f$. 

Our motivating example will be a general curve of genus $12$ and degree $14$ in $\PP^4$. 

\begin{proposition}
	\label{expected1412}
	Let $C$ be a general linearly normal non-degenerate curve of genus $12$ and degree $14$ in $\PP^4$. Then $C$ is of maximal rank, and the homogeneous coordinate ring $S_C=S/I_C$ and the section ring $\Gamma_*(\sO_C):=\oplus_{n \in \ZZ}\HH^0(\sO_C(n))$ have minimal free resolutions with the following Betti tables:
	\begin{equation}
	\label{res1412}
\bettif{
	0& 	1 &		.  & 	  &		  & 	 \\
	1& 	  &		.  & 	  &		  & 	 \\
	2& 	  &		4 & 	  &		  & 	 \\
	3& 	  &		5 & 	18&		12& 	2 
	}	
\qquad \qquad
\bettit{
0&	    1  & 	 &		 & 	 		   	 \\
1& 	.  &		  & 	  &		  	 \\
2& 	2  &	14	 & 15	  & 2		   	 \\
3& 	  &		 & 	&		 2
}
	\end{equation}
	In particular, the cubic threefolds containing $C$ form a $\PP^3$. The minimal resolution of $\Gamma_*(\sO_C)$ as a module over the homogeneous coordinate ring of a  cubic threefold $X \supset C$ is eventually 2-periodic with Betti numbers
	\[
\begin{array}{c|ccccccc}
& 	\rule{0.5ex}{0pt} 0 \rule{0.5ex}{0pt} &	\rule{0.5ex}{0pt} 1 \rule{0.5ex}{0pt} & \rule{0.5ex}{0pt} 2 \rule{0.5ex}{0pt} &		\rule{0.5ex}{0pt} 3 \rule{0.5ex}{0pt} & \rule{0.5ex}{0pt} 4 \rule{0.5ex}{0pt} &  \dotso \\ \hline
0&	    1  & 	 	  &		    &		   	 \\
1& 	.  &		  & 	    &		  	 \\
2& 	2  &	13	  & 15	    & 2		   	 \\
3& 	   &		  & 2		& 15 & 15 & \dotso\\
4&        & & & & 2 & \dotso
\end{array}
	\]
	
	\begin{proof}
We assume that the maps $\HH^0(\PP^4,\sO_{\PP^4}(n))\rightarrow \HH^0(\PP^4,\sO_{C}(n))$ are of maximal rank, i.e., $C$ has maximal rank. Since $\sO_C(n)$ is non-special for $n \geq 2$, by Riemann--Roch we can compute the Hilbert function of the homogeneous coordinate ring of $C$ and therefore the numerator of its Hilbert series
\[
(1-t)^5H_C(t)=1-4t^3-5t^4+18t^5-12t^6+2t^7.
\]
Thus, we expect the Betti table of $S/I_C$ to look like the one in (\ref{res1412}). Analogously, the numerator of the Hilbert series of $\Gamma_*(\sO_C)$ under the maximal rank assumption is
\[
(1-t)^5H_{\Gamma_*(\sO_C)}(t)=1+2t^2-14t^3+15t^4-2t^5-2t^6
\]
and the expected Betti table is (\ref{res1412}).

To show that the Betti tables are indeed the expected ones and that, a posteriori, a general curve $C$ is of maximal rank, we only need to exhibit a concrete example, which we construct via matrix factorizations as explained in the proof of Theorem \ref{unirationalityThm} and summarized in Algorithm \ref{algorithmUnirat}. The function \texttt{verifyAssertionsOfThePaper(1)} of \cite{SchreyerTanturriCode} produces the \Mac code needed to verify all the above assertions. Another family of examples can be obtained as explained in Corollary \ref{unirulednessThm}.

A free resolution of $\Gamma_*(\sO_C)$ as a module over the cubic hypersurface ring $S_X$ can be obtained via the Shamash construction, from which we can deduce the Betti numbers of the minimal $S_X$-resolution: 
$$\beta^{S_X}_{1,3}(\Gamma_*\sO_C)=\beta^S_{1,3}(\Gamma_*\sO_C)-1$$
 since the equation of $X$ is superfluous over $S_X$, and $\beta^{S_X}_{2,5}(\Gamma_*\sO_C)=\beta^{S_X}_{3,5}(\Gamma_*\sO_C)=2$ follows from (\ref{shamashzero}).
	\end{proof}
\end{proposition}

\begin{remark}
\label{posCharSuffices}
Throughout the paper we will sometimes need to exhibit explicit examples of modules defined over the rationals $\QQ$ or complex numbers $\CC$ satisfying some open conditions on their Betti numbers. Our constructions will involve only linear algebra, especially Gr\"obner basis computations, and will depend only on the choice of some parameters; a choice of rational values for the parameters thus gives rise to modules over $\QQ$, hence over $\CC$.
An ultimate goal would be to perform the computations over the function field $\QQ(t_1,\ldots,t_N)$, where $N$ is the number of free parameters.
This however is out of reach for computer algebra systems today.


We have implemented our constructions using the computer algebra system \Mac \cite{M2}. A priori it would be possible to perform these computations over $\QQ$, but this might require too much time, so instead we work over a finite prime field $\mathbb{F}_p$. We can view our choice of the initial parameters in $\mathbb{F}_p$ as the reduction modulo $p$ of some choices of parameters in $\ZZ$. Then, the so-obtained module $M_p$ can be seen as the reduction modulo $p$ of a family of modules defined over a neighborhood $\Spec \ZZ[\frac{1}{b}]$ of $(p) \in \Spec \ZZ$ for a suitable $b \in \ZZ$ with $ p \nmid  b$.
If $M_p$ satisfies our open conditions, then by semicontinuity the generic fiber $M$ satisfies the same open conditions, and so does the general element of the family over $\QQ$ or $\CC$.
\end{remark}

Let $C$ be a curve as in Proposition \ref{expected1412}. We can consider $M=\Gamma_*(\sO_C)$ as a $S_X$-module, being $X$ a generally chosen cubic threefold containing $C$. If $C$ is general, the periodic part of its minimal free resolution yields, up to twist, a matrix factorization of the form
\[
\xymatrix{
	S^{15} \oplus S^2(-1) &
	S^2(-1) \oplus S^{15}(-2)
	\ar[l]_-{\psi} &
	S^{15}(-3) \oplus S^{2}(-4) \ar[l]_-\varphi.
%
%
}
\]

\begin{definition}[Shape of a matrix factorization]
We will call the Betti numbers of the minimal  periodic resolution 
\begin{equation*}
\begin{array}{ccccc}
15 & 2 & & \\
2 & 15 & 15 & 2 \\
& & 2 & 15 & \ldots \\
\end{array}
\end{equation*}
the \emph{shape} of the matrix factorization. When the degree $s$ of the hypersurface containing the curve is fixed (in the current example we have $s=3$), then the shape of a matrix factorization is determined by the Betti numbers  $\beta(\psi)$ of $\psi$. In the current case they are
\begin{equation}
\label{shape1412}
\begin{array}{cc}
15 & 2 \\
2 & 15
\end{array}
\end{equation}
\end{definition}

In general, starting from a curve $C$ in $\PP^4$ contained in a (smooth) hypersurface $X$, the 2-periodic part of a minimal resolution of the section module $\Gamma_*(\sO_C)$ over $S_X$ will produce a matrix factorization. The shape is uniquely determined for a general pair $C \subset X \subset \PP^4$ in a component of the Hilbert scheme of pairs. For a given pair, different choices of the resolution yield equivalent  matrix factorizations. They all define the same sheaf $\sF=(\coker \varphi)^\sim$ on $X$, which turns out to be an ACM vector bundle, see e.g.\ \cite[Proposition 2.1]{CasanellasHartshorneGorenstein}.

We have thus established one way of the correspondence between curves and matrix factorizations. In what follows we will see that, to some extent, it is possible to recover the original curve from the matrix factorizations it induces.

\subsection{Monads and the Reconstruction Theorem}

Let us consider a pair $(C,X)$ of a general curve $C$ of degree 14 and genus 12 in $\PP^4$ and a general (smooth) cubic hypersurface $X=\Vi(f)\supset C$. The curve induces, up to twist, a matrix factorization of shape (\ref{shape1412})
\begin{equation*}
{\small
\xymatrix{
	\sO_X^{15}(-1) \oplus \sO_X^2(-2) &
	\ar[l]_-\psi \sO_X^2(-2) \oplus \sO_X^{2+13}(-3) &
	\ar[l]_-\varphi \sO_X^{15}(-4) \oplus \sO_X^{2}(-5).
}
}
\end{equation*}
Here, we have distinguished in $\sO_X^{2+13}(-3)$ the two copies coming directly from the third step of the resolution of $\Gamma_*{\sO_C}$ as an $S$-module, see the Shamash construction (\ref{shamashConstr}). The map $\psi$ can be regarded as a block matrix, with a zero submatrix $\sO_X^2(-2)  \leftarrow  \sO_X^2(-2) \oplus \sO_X^{2}(-3)$ by (\ref{shamashzero}).

Let $\sF = (\coker \varphi)^{\sim}$; we can form a complex
\begin{equation}
\label{monad812}
\xymatrix{
	0  &\ar[l]
	\sO_X^2(-2)  & \ar[l]_{\qquad \alpha}
	\sF  & \ar[l]_{\beta  \phantom{abcdefg1}}
	\sO_X^{2}(-2) \oplus \sO_X^{2}(-3)  & \ar[l]
	0.
	}
\end{equation}
We claim that this complex is a monad for the ideal sheaf $\sI_{C/X}$, i.e., $\alpha$ is surjective, $\beta$ injective and $\ker \alpha / \image \beta \cong \sI_{C/X}$. In other words, we can recover the original curve $C$ from the complex. The claim is a special case of the following

\begin{theorem}[Reconstruction Theorem]
	\label{reconstructionThm}
Let $C\subset \PP^4$ be a non-degenerate linearly normal curve of genus $g$ and degree ${d} \ge g$ not contained in any quadric
and let $X=\Vi(f)$ be a smooth hypersurface of degree $s$ containing $C$. Let $F_{\bullet}$ and $\overline G_{\bullet}$ be minimal free resolutions of
 $\Gamma_*(\sO_C)$ over $S$ and $S/f$ respectively, let $\varphi$ denote the syzygy map $\overline G_3 \leftarrow \overline G_4$ and $\sF=(\coker \varphi)^{\sim}(s)$. Then the complex of vector bundles on $X$
\begin{equation}
\label{monad}
\xymatrix{
	0 &\ar[l] 
	(\overline{F'_0})^{\sim} &\ar[l]_{\phantom{abcd}\alpha} 
	\sF &\ar[l]_{\beta \phantom{ab}} 
	\left(\overline{F_3}(s)\right)^{\sim} &\ar[l] 
	0,
}
\end{equation}
where the maps are induced by $\overline G_{\bullet}$ via the Shamash construction and $F'_0$ is the complement of $S$ in $F_0=S \oplus F'_0$,  is a monad for the ideal sheaf of $C$ on $X$, i.e., $\beta$ is injective, $\alpha$ is surjective, and $\ker \alpha/\image \beta \cong \sI_{C/X}$.
If $s\ge 4$ the monad is uniquely determined by $\sF$.
\end{theorem}

\begin{proof}
Since ${d} \ge g$ the line bundle $\sO_C(2)$ is non-special. It follows that $\Gamma_*(\sO_C)$ has Betti table
$$
\bettit{
0&	    1  & 	 &		 & 	 		   	 \\
1& 	.  &		  &   &		  	 \\
2& 	\beta_{0,2}  &	\beta_{1,3}& \beta_{2,4}  & \beta_{3,5}	   	 \\
3& 	\beta_{0,3}  &	\beta_{1,4}	 & \beta_{2,5}	&		 \beta_{3,6} \\
}
$$
Indeed, $\beta_{1,2}=0$ by assumption. Since ${\rm Hom}(F_{\bullet},S(-5))$ resolves $\Gamma_*(\omega_C)$, we must have $\beta_{3,n}=0$ for
$n-5\ge 2$, because $\HH^0(\omega_C(-2))=0$. So $\Gamma_*\sO_C$ is 3-regular and non-zero Betti numbers can only occur in the indicated range. 

Let us assume $s=3$. The Shamash resolution starts with the Betti numbers
$$
\bettit{
0 & 1 & \phantom{\beta_{3,6}  + \beta_{1,3}} & \phantom{\beta_{3,6}  + \beta_{1,3}}	& \phantom{\beta_{3,6}  + \beta_{1,3}} \\
1& \phantom{\beta_{3,6}  + \beta_{1,3}}	                   &	                          &  \phantom{\beta_{2,4} +}\phantom{a} 1  \phantom{ab} 	                             &		  	 \\
2& 	\beta_{0,2}  &	\beta_{1,3}         & \beta_{2,4}  + 0 \phantom{ab}        & \beta_{3,5} 	\phantom{ + \beta_{1,2} a}   	 \\
3& 	\beta_{0,3}  &	\beta_{1,4}	 & \beta_{2,5}+\beta_{0,2}	& \beta_{3,6}  + \beta_{1,3}\\
4&                       &                              & 0 \phantom{ab} +                \beta_{0,3}&  0 \phantom{ab} +                \beta_{1,4}     \\
}
$$
We see that, in the induced map $\overline{F_1}\leftarrow\overline{F_0}(-3)$, there is a non-zero constant $1\times 1$ submatrix; this means that in this case the Shamash resolution is always non-minimal, and in a minimal resolution a cancellation occurs, causing $\beta_{1,3}$ to decrease by one. Such cancellation corresponds to the equation $f$ of $X$ in $S$ becoming superfluous in $S/f$.

By definition, the map $\overline G_2 \leftarrow \overline G_3$ factorizes through $\sF$. As
\[
F_0=S \oplus S^{\beta_{0,2}}(-2) \oplus S^{\beta_{0,2}}(-3), \qquad
F_3=S^{\beta_{3,5}}(-5) \oplus S^{\beta_{3,6}}(-6), 
\]
 the complex (\ref{monad}) has the form
$$
 0 \leftarrow \sO^{\beta_{0,2}}_X(-2)\oplus \sO^{\beta_{0,3}}_X(-3) \leftarrow \sF \leftarrow \sO^{\beta_{3,5}}_X(-2)\oplus \sO^{\beta_{3,6}}_X(-3) \leftarrow 0.
$$
It is indeed a complex because of (\ref{shamashzero}). We claim that the first map is surjective, the second one is injective and that the homology in the middle is isomorphic to $\sI_{C/X}$.

The first claim follows since the cokernel of the composition
$$
\sO^{\beta_{0,2}}_X(-2)\oplus \sO^{\beta_{0,3}}_X(-3) \leftarrow \sF \leftarrow \sO^{\beta_{1,3}-1}_X(-3)\oplus \sO^{\beta_{1,4}}_X(-3),
$$
where the ``${}-1$'' represents the missing equation of $X$ over $S/f$, coincides by construction with the sheafification restricted to $X$ of the cokernel of $F'_0 \leftarrow F_1$; such cokernel is a module of finite length (a submodule of the Hartshorne--Rao module of $C$), hence its sheafification is zero.

Let $\sG:=\ker(\alpha)$. Being the sheafification of a MCM module over $X$, the sheaf $\sF$ is a vector bundle and $\sG$ is a vector bundle as well. It remains to show that
$$
\xymatrix{\sG &\ar[l]_-{\gamma} \sO^{\beta_{3,5}}_X(-2)\oplus \sO^{\beta_{3,6}}_X(-3)
}
$$ 
is injective and a presentation of $\sI_{C/X}$. To see this, we apply the functor $\sHom(-, \omega_X)$ to $\gamma$ and obtain  $$
\xymatrix{\sG^*(-2) \ar[r]^-{\gamma^*(-2)} & \sO^{\beta_{3,5}}_X\oplus \sO^{\beta_{3,6}}_X(1).
}
$$ 
The cokernel of this map coincides by construction with the cokernel of the dual of the sheafification of the last map of $F_{\bullet}$
$$
\xymatrix{
\sO_X^{\beta_{2,4}}(-1)\oplus \sO_X^{\beta_{2,5}} \ar[r] & \sO_X^{\beta_{3,5}} \oplus \sO_X^{\beta_{3,6}}(1),
}
$$
which is a presentation of $\omega_C$ by duality on $\PP^4$.
Since 
\begin{align*}
\rank \sF &= \rank F_0-\rank F_1+\rank F_2+\rank F_0\cr 
               &=\rank F_3+\rank F'_0+1,
\end{align*}
we have $\rank \sG = \beta_{3,5} + \beta_{3,6} +1$. Hence both $\gamma^*(-2)$ and $\gamma$ drop rank in expected codimension $2$; applying again $\sHom(-, \omega_X)$ to $\gamma^*(-2)$ we get that $\gamma$ is injective and by the Hilbert--Burch Theorem \cite[Theorem 20.15]{Eisenbud} it fits into an exact complex
\[
\xymatrix{
	0 &
	\ar[l]  \sO_C(\ell) &
	\ar[l] \sO_X(\ell) &
	\ar[l]  \sG   &
	\sO^{\beta_{3,5}}_X(-2)\oplus \sO^{\beta_{3,6}}_X(-3) \ar[l]_-{\gamma} &
	0 \ar[l]
}
\]
for some $\ell$. By applying again $\sHom(-, \omega_X)$ to this last exact sequence one gets that $\gamma^*(-2)$ is a presentation of $\omega_C(-\ell)$, hence $\ell=0$.

The argument for  $s\ge 4$ is similar, the only difference being that the second and third term in the Shamash resolution of $\Gamma_*(\sO_C)$ differ in their twist. For example, the third term is
$S^{\beta_{3,5}}(-5)\oplus S^{\beta_{3,6}}(-6) \oplus S^{\beta_{1,3}}(-3-s)\oplus S^{\beta_{1,4}}(-4-s)$, and we see that for $s\ge 4$ the monad is uniquely determined by $\sF$.
\end{proof}

\section{General matrix factorizations and uniruledness}	
\label{uniruledness}
	
In the last section we saw how, from a matrix factorization induced by a curve $C$, it is possible to recover $C$ itself. Within this section we will show how we can use the Reconstruction Theorem to actually \emph{construct} new curves on a hypersurface $X$, starting from a general matrix factorization $(\psi,\varphi)$ on $X$
. The key point for proving this result is exhibiting, case by case in the range of interest of the paper
\[
(g,{d}) \in \{(12,14), (13,15), (16,17), (17,18), (18,19), (19,20), (20,20)\},
\] a concrete example satisfying some open conditions.

This leads naturally to the strategy of constructing (unirational) families of matrix factorizations on a hypersurface $X$ to approach the problem of constructing projective curves. In the range of interest for this paper, such strategy turns out to be effective because of the following considerations.

On the one hand, one could have that a general hypersurface of the appropriate degree does not contain a curve with the prescribed genus and degree. As anticipated in the introduction and proved in the last section, Theorem \ref{generalQuartic} shows that this does not happen, so we can start with a general choice of $X$.

On the other hand, the space of matrix factorizations of a given shape on a general hypersurface may very well have many components. This leads to the following
\begin{question}
	Is the space of matrix factorizations of shape (\ref{shape1412}) on a general cubic hypersurface irreducible?
\end{question}
Other similar questions can of course be formulated for other cases of interest. Our approach will be to construct unirational families of matrix factorizations, dominants on some component (or union of components) of such space; we will then show that the points in such particular component give rise to the desired curves. This last claim requires further explanations, see Remark \ref{grassmannians} here below.

\subsection{Constructing new curves from matrix factorizations}

\begin{remark}
	\label{grassmannians}
	Let $(\psi, \varphi)$ be a general matrix factorization of shape (\ref{shape1412}) over a cubic hypersurface $X$; in particular, we have a map
	\[
	\xymatrix{
		\sO_X^{15}(-1) \oplus \sO_X^2(-2) &
		\ar[l]_-\psi \sO_X^2(-2) \oplus \sO_X^{15}(-3).
	}
	\]
	In Theorem \ref{reconstructionThm} we constructed, from a matrix factorization induced by a curve, a complex (\ref{monad812}). To find a similar complex we need to find a rank 2 subbundle $\sO_X^2(-3)$ of $\sO_X^{15}(-3)$ such that the composition
	\[
	\xymatrix{
		\sO_X^2(-2) &
		\ar[l]_-\delta\sO_X^{15}(-3) & 
		\ar@{_{(}->}[l] \sO_X^2(-3),
	}
	\]
	where $\delta$ is the map extracted from $\psi$, is zero. The map $\delta$ is represented by a matrix of linear forms, and has a kernel of dimension at least $5=15-2\cdot \hh^0(\mathcal{O}_{\mathbb{P}^4}(1))$. Having such kernel of dimension precisely $5$ is an open condition on the space of matrix factorizations which is satisfied by the concrete examples we construct in \cite{SchreyerTanturriCode}. This means that, for a given general matrix factorization, we get a complex for any choice of $\sO_X^2(-3)$ inside $\ker \delta$, which corresponds to the choice of a $p \in \GG(2,5)$. A general choice of $p$ produces a complex (\ref{monad812}); Theorem \ref{constructionThm} below will show that such complex is a monad for a smooth curve of genus 12 and degree 14.
	
	The situation is very similar in the case of curves of genus 13 and degree 15. Here we have to choose again a rank 2 subbundle $\sO_X^2(-3)$ inside the kernel, which is now $3$-dimensional in general. This yields many choices parameterized by $\GG(2,3)=\PP^2$. Again, a general choice produces a monad and a smooth curve.
\end{remark}	

\begin{theorem}
	\label{constructionThm}
	Let $(g,{d})$ be in
	\[
	\{(12,14), (13,15), (16,17), (17,18), (18,19), (19,20), (20,20)\}
	\]
	and let ${H}_{{d},g}$ be the component of the Hilbert scheme $\Hilb_{{d} t+1-g}(\PP^4)$ dominating $\sM_g$. Let $C \in {H}_{{d},g}$ be a general point, i.e., a general curve of genus $g$ and degree ${d}$ in $\PP^4$.
	\begin{enumerate}
	\item The quotient ring $S/I_C$ and the section module $\Gamma_*(\sO_C)$ have expected resolutions, i.e., their Betti tables correspond to the ones listed in Table \ref{expectedresolutions} below.
	\item Let $s = \min \{s' \mid \hh^0(\sI_C(s')) \neq 0 \}$ and consider a general hypersurface $X$ with equation $f \in (I_C)_s$. The minimal free $S/f$-resolution of $\Gamma_*(\sO_C)$ is eventually 2-periodic and gives rise to a matrix factorization of $f$ of shape as in Table \ref{shapesmonads}
.
	\item For each choice of $(g,{d})$ above, let $s$ be the (expected) minimum degree of a hypersurface containing a general curve of genus $g$ and degree ${d}$ and let $X$ be a general hypersurface of degree $s$. There is a component of the space of matrix factorizations on $X$ of shape corresponding to $(g,{d})$ in Table \ref{shapesmonads} whose general element gives rise to complexes of the form (\ref{monad}), which turns out to be a monad for $\sI_{C'/X}$, the ideal sheaf of a smooth curve $C'$ of genus $g$ and degree ${d}$ with respect to $X$.
	\end{enumerate}
	\begin{proof}
	As  in Proposition \ref{expected1412}, we can compute the expected Betti tables of the $S$-resolutions of $S/I_C$ and $\Gamma_*(\sO_C)$. These are summarized in Table \ref{expectedresolutions}. In Table \ref{shapesmonads} we list the expected shapes of the matrix factorizations and the corresponding monads we can construct.
	
	For a matrix factorization, giving rise to a monad for the ideal sheaf of a smooth curve with right genus and degree is an open condition. When the complex is not uniquely determined, i.e.\ for $s=3$ (see Remark \ref{grassmannians}), it is an open condition on the space of complexes, parametrized by a rational variety. To prove the third part of the Theorem, it is thus sufficient to explicitly construct, for each of the aforementioned cases, a matrix factorization of the given shape and a complex of the form (\ref{monad}) which is a monad for a smooth curve with assigned genus and degree. The fact that a general hypersurface $X$ of the appropriate degree $s$ contains such a curve will be proved in Theorem \ref{generalQuartic} and relies again on the computation of explicit examples.
	
	The function \texttt{verifyAssertionsOfThePaper(2)} of \cite{SchreyerTanturriCode} provides the \Mac code useful to
	produce, for each pair $(g,{d})$, a matrix factorization on a hypersurface $X$ of degree $s$ such that
	\begin{itemize}
	\item the shape of the matrix factorization is as listed in Table \ref{shapesmonads};
	\item a complex built from the matrix factorization, according to the numerology of the expected resolution of the section module of a general curve and the Reconstruction Theorem \ref{reconstructionThm}, is a monad for a smooth curve $C$ of genus $g$ and degree ${d}$;
	\item $S/I_C$ and $\Gamma_*{\sO_C}$ have expected resolutions as in Table \ref{expectedresolutions}, and $\Gamma_*{\sO_C}$ induces a matrix factorization on a general supporting hypersurface $X'$ of degree $s$ of shape as in Table \ref{shapesmonads}.
	\end{itemize}
	
	To prove the first two points of the Theorem, which correspond to open conditions on ${H}_{{d},g}$, it is sufficient to check the last assertion on a particular example.
	
%
	We use different constructions to explicitly exhibit a matrix factorization satisfying the statements. For $g=12$ or $g=13$, the procedure followed can be found in Corollary \ref{unirulednessThm}. For $g=12$, an alternative way is to use curves of genus 10 and degree 13, as explained in Proposition \ref{auxiliaryE}. For $g \geq 16$, see Section \ref{familiesOfCurves}. As mentioned in Remark \ref{posCharSuffices}, it is sufficient to run our constructions over a finite field.\qedhere
	{\small
	\begin{table}[h!bt]
	\caption{Expected Betti tables.}
	\label{expectedresolutions}
	\begin{tabular}{ccc} \toprule
		$(g,{d})$	 &	$\beta_{i,j}(S/I_C)$ & $\beta_{i,j}(\Gamma_*(\sO_C))$\\ \midrule
		\rule{1.2ex}{0ex}$(12,14)$\rule{1.2ex}{0ex} & \rule{1.2ex}{0ex}$\bettif{
			0& 	1 &		.  & 	  &		  & 	 \\
			1& 	  &		.  & 	  &		  & 	 \\
			2& 	  &		4 & 	  &		  & 	 \\
			3& 	  &		5 & 	18&		12& 	2 
		}$\rule{1.2ex}{0ex} & 
		\rule{1.2ex}{0ex}$\bettit{
			0&	    1  & 	 &		 & 	 		   	 \\
			1& 	.  &		  & 	  &		  	 \\
			2& 	2  &	14	 & 15	  & 2		   	 \\
			3& 	  &		 & 	&		 2
		}$ \rule{1.2ex}{0ex}
		 \\ \midrule
		\rule{1.2ex}{0ex}$(13,15)$\rule{1.2ex}{0ex} & \rule{1.2ex}{0ex}$\bettif{
			0& 	1 &		.  & 	  &		  & 	 \\
			1& 	  &		.  & 	  &		  & 	 \\
			2& 	  &		2 & 	  &		  & 	 \\
			3& 	  &		12 & 	27&		17& 	3 
		}$\rule{1.2ex}{0ex} & 
		\rule{1.2ex}{0ex}$\bettit{
			0&	    1  & 	 &		 & 	 		   	 \\
			1& 	.  &		  & 	  &		  	 \\
			2& 	3  &	17	 & 18	  & 3		   	 \\
			3& 	  &		 & 	&		 2
		}$ \rule{1.2ex}{0ex}
		\\ \midrule
		\rule{1.2ex}{0ex}$(16,17)$\rule{1.2ex}{0ex} & \rule{1.2ex}{0ex}$\bettif{
			0& 	1 &		.  & 	  &		  & 	 \\
			1& 	  &		.  & 	  &		  & 	 \\
			2& 	  &		. & 	  &		  & 	 \\
			3& 	  &		17 & 	29&		13& 	 \\
			4& 	  &		 & 	&		1& 	1 
		}$\rule{1.2ex}{0ex} & 
		\rule{1.2ex}{0ex}$\bettit{
			0&	    1  & 	 &		 & 	 		   	 \\
			1& 	.  &		  & 	  &		  	 \\
			2& 	4  &	19	 & 18	  & 1		   	 \\
			3& 	  &		 & 	&		 3
		}$ \rule{1.2ex}{0ex}
		\\ \midrule	
		\rule{1.2ex}{0ex}$(17,18)$\rule{1.2ex}{0ex} & \rule{1.2ex}{0ex}$\bettif{
			0& 	1 &		.  & 	  &		  & 	 \\
			1& 	  &		.  & 	  &		  & 	 \\
			2& 	  &		. & 	  &		  & 	 \\
			3& 	  &		14 & 	18&		& 	 \\
			4& 	  &		 & 2	&		10& 	3 
		}$\rule{1.2ex}{0ex} & 
		\rule{1.2ex}{0ex}$\bettit{
			0&	    1  & 	 &		 & 	 		   	 \\
			1& 	.  &		  & 	  &		  	 \\
			2& 	5  &	22	 & 21	  & 2		   	 \\
			3& 	  &		 & 	&		 3
		}$ \rule{1.2ex}{0ex}
		\\ \midrule
		\rule{1.2ex}{0ex}$(18,19)$\rule{1.2ex}{0ex} & \rule{1.2ex}{0ex}$\bettif{
			0& 	1 &		.  & 	  &		  & 	 \\
			1& 	  &		.  & 	  &		  & 	 \\
			2& 	  &		. & 	  &		  & 	 \\
			3& 	  &		11 & 	7&		& 	 \\
			4& 	  &		 & 	     17 &		19& 	5  
		}$\rule{1.2ex}{0ex} & 
		\rule{1.2ex}{0ex}$\bettit{
			0&	    1  & 	 &		 & 	 		   	 \\
			1& 	.  &		  & 	  &		  	 \\
			2& 	6  &	25	 & 24	  & 3		   	 \\
			3& 	  &		 & 	&		 3
		}$ \rule{1.2ex}{0ex}
		\\ \midrule
		\rule{1.2ex}{0ex}$(19,20)$\rule{1.2ex}{0ex} & \rule{1.2ex}{0ex}$\bettif{
			0& 	1 &		.  & 	  &		  & 	 \\
			1& 	  &		.  & 	  &		  & 	 \\
			2& 	  &		. & 	  &		  & 	 \\
			3& 	  &		8 & 	&		& 	 \\
			4& 	  &		4 & 	32&		28& 	7 
		}$\rule{1.2ex}{0ex} & 
		\rule{1.2ex}{0ex}$\bettit{
			0&	    1  & 	 &		 & 	 		   	 \\
			1& 	.  &		  & 	  &		  	 \\
			2& 	7  &	28	 & 27	  & 4		   	 \\
			3& 	  &		 & 	&		 3
		}$ \rule{1.2ex}{0ex}
		\\ \midrule
		\rule{1.2ex}{0ex}$(20,20)$\rule{1.2ex}{0ex} & \rule{1.2ex}{0ex}$\bettif{
			0& 	1 &		.  & 	  &		  & 	 \\
			1& 	  &		.  & 	  &		  & 	 \\
			2& 	  &		. & 	  &		  & 	 \\
			3& 	  &		9 & 	.&		& 	 \\
			4& 	  &		 & 26	&	24& 	6
		}$\rule{1.2ex}{0ex} & 
		\rule{1.2ex}{0ex}$\bettit{
			0&	    1  & 	 &		 & 	 		   	 \\
			1& 	.  &		  & 	  &		  	 \\
			2& 	6  &	24	 & 21	  & .		   	 \\
			3& 	  &		 & 	&		 4
		}$ \rule{1.2ex}{0ex}
		 \\ \bottomrule
	\end{tabular}
	\end{table}
}
	
	{\small
	\begin{table}[h!bt]
		\caption{Shapes of the matrix factorizations and corresponding monads.}
		\label{shapesmonads}
		\begin{tabular}{ccc} \toprule
			$(g,{d})$	 &	shape of $\psi$ & monad\\ \midrule
			$(12,14)$ & 
			$\begin{array}{cc}
				15 & 2 \\
				2 & 15
			\end{array}$ & 
			$\xymatrix{
				\sO_X^2(-2) \oplus \sO_X^{2}(-3) \ar@{^{(}->}[r] &
				\sF \ar@{->>}[r] &
				\sO_X^{2}(-2)
			}$
			\\ \midrule
			$(13,15)$ & 
			$\begin{array}{cc}
			18 & 3 \\
			3 & 18
			\end{array}$ & 
			$\xymatrix{
				\sO_X^3(-2) \oplus \sO_X^{2}(-3) \ar@{^{(}->}[r] &
				\sF \ar@{->>}[r] &
				\sO_X^{3}(-2)
			}$
			\\ \midrule
			$(16,17)$ & 
			$\begin{array}{cc}
			19 & 1 \\
			. & 3 \\
			4 & 19			
			\end{array}$ & 
			$\xymatrix{
				\sO_X(-1) \oplus \sO_X^{3}(-2) \ar@{^{(}->}[r] &
				\sF \ar@{->>}[r] &
				\sO_X^{4}(-2)
			}$
			\\ \midrule
			$(17,18)$ & 
			$\begin{array}{cc}
			22 & 2 \\
			. & 3 \\
			5 & 22
			\end{array}$ & 
			$\xymatrix{
				\sO_X^2(-1) \oplus \sO_X^{3}(-2) \ar@{^{(}->}[r] &
				\sF \ar@{->>}[r] &
				\sO_X^{5}(-2)
			}$
			\\ \midrule
			$(18,19)$ & 
			$\begin{array}{cc}
			25 & 3 \\
			. & 3 \\
			6 & 25
			\end{array}$ & 
			$\xymatrix{
				\sO_X^3(-1) \oplus \sO_X^{3}(-2) \ar@{^{(}->}[r] &
				\sF \ar@{->>}[r] &
				\sO_X^{6}(-2)
			}$
			\\ \midrule
			$(19,20)$ & 
			$\begin{array}{cc}
		    28 & 4 \\
		    . & 3 \\
		    7 & 28
			\end{array}$ & 
			$\xymatrix{
				\sO_X^4(-1) \oplus \sO_X^{3}(-2) \ar@{^{(}->}[r] &
				\sF \ar@{->>}[r] &
				\sO_X^{7}(-2)
			}$
			\\ \midrule
			$(20,20)$ & 
			$\begin{array}{cc}
			22 & . \\
			. & 4\\
			6 & 24
			\end{array}$ & 
			$\xymatrix{
				\sO_X^{4}(-2) \ar@{^{(}->}[r] &
				\sF \ar@{->>}[r] &
				\sO_X^{6}(-2)
			}$
			\\ \bottomrule
		\end{tabular}
	\end{table}
}
	\end{proof}
	\end{theorem}
	
\begin{remark}
		Theorem \ref{constructionThm} holds also in the case of curves of genus 15 and degree 16; the study of that particular case allowed the first author to construct some unirational families of such curves and to show the uniruledness of $\sW^4_{16,15}$ \cite{SchreyerMatrix}. The case of genus 16 and degree 17 was already the topic of  the master's thesis \cite{MuellerThesis}. 
\end{remark}

	\begin{remark}
		We expect Theorem \ref{constructionThm} to hold in other circumstances as well. Our interest in the cases above has the following reasons.
		
		The first two cases correspond to the Brill--Noether spaces $\sW^4_{12,14}$ and $\sW^4_{13,15}$, which by Serre duality are birational to $\sW^1_{12,8}$ and $\sW^1_{13,9}$ respectively. 
		
		The remaining cases are motivated by a (so far unsuccessful) attempt of proving the unirationality of the moduli space $\sM_g$ for $g \geq 16$. We have chosen ${d}$ such that $\rho(g,4,{d})=g - 5\hh^1(\sO_C(1))$ takes the minimal non-negative value. See Section \ref{families} for further details.
		
		There are cases in which we do not expect the Theorem to hold, at least not in the formulation above. For instance, consider 
 the family of curves of genus 14 and degree 16 in $\PP^4$ which are contained in cubic hypersurfaces. These curves forms a divisor $\sD$ in 
 $\sW^4_{14,16}$. Their matrix factorizations have the shape
		\[
		\begin{array}{cc}
		21 & 4 \\
		4 & 21
		\end{array}
		\]
and we would need  a rank 2 subbundle inside the kernel of the map corresponding to the last row of the Betti table above. As the kernel of a general such map is just 1-dimensional, we believe that a general matrix factorization of this shape is not  induced by any curve in $\sD$. 
	\end{remark}

\subsection{Uniruledness results}
	
A consequence of Remark \ref{grassmannians} and of Theorem \ref{constructionThm} is that, if $(g,{d})=(12,14)$ or $(13,15)$, for a fixed general matrix factorization (in the sense of Theorem \ref{constructionThm}) on a general cubic hypersurface $X$ of shape as in Table \ref{shapesmonads} we have a rational map
\begin{equation}
\label{rationalChoice}
\xymatrix{
	\sV \ar@{-->}[r] & \sW^4_{g,{d}},
}
\end{equation}
where $\sV$ is $\GG(2,5)$, $\PP^2$ respectively.

\begin{corollary}\label{unirulednessThm}
$\sW^4_{12,14}$ and $\sW^4_{13,15}$ and the corresponding $\sW^1_{12,8}$ and $\sW^1_{13,9}$ are uniruled.
	\end{corollary}
\begin{proof} 
Take a general point in $\sW^4_{12,14}$ or $\sW^4_{13,15}$ and choose an embedding $C \subset \mathbb{P}^4$. Consider a general cubic hypersurface $X$ containing $C$ and consider the induced matrix factorization on $X$. The induced rational map (\ref{rationalChoice}) sends a general choice of the monad to a point $\sW^4_{12,14}$, $\sW^4_{13,15}$ respectively. The image of this map is a rational variety; if it is not a point, then it contains a rational curve which passes through $C$ and whose points parametrize points of $\sW^4_{12,14}$, $\sW^4_{13,15}$ respectively, whence the conclusion. For the map (\ref{rationalChoice}), being non-constant is an open condition on the space of matrix factorizations, hence it is sufficient to check it for a concrete example.

To construct the two necessary examples, we start from a $g$-nodal rational curve $C'$ of genus $g$ having a $g^1_{2g-2-{d}}=|D|$ (see \cite{BoppMaster,BoppCode}). We embed $C'$ in $\PP^4$ via $|K_{C'}-D|$ and obtain a singular curve $C' \subset \PP^4$ of genus $g$ and degree ${d}$. We consider  the matrix factorization on a  cubic hypersurface obtained from $C'$ and choose a random point in $\sV$.  We check that the resulting curve $C$ is smooth; since $C'$ is a point in the boundary as a point in $\overline {\sW}^4_{d,g}$, the map is not constant. An implementation of the code is provided by the function \texttt{verifyAssertionsOfThePaper(2)} in \cite{SchreyerTanturriCode}.

By passing to the Serre dual linear systems, this yields the uniruledness of the corresponding spaces $\sW^1_{12,8}$ and $\sW^1_{13,9}$ as well.  \qedhere
	\end{proof}

\section{A unirational Hurwitz space}
\label{uniratHurwitz}

Our aim is to use all the machinery developed so far to construct a unirational family of curves dominating $H_{12,14}$, the component of the Hilbert scheme of curves of genus 12 and degree 14 in $\PP^4$ which dominates $\sW^4_{12,14}$. By considering the dual models, this will imply the unirationality of $\sW^1_{12,8}$ and $\sH_{12,8}$.

The idea is to use Theorem \ref{constructionThm}. If we manage to produce a large enough unirational family of general matrix factorizations, we can hope that the space of curves we obtain is dominant. In other terms, we translate the problem of constructing curves with fixed invariants to the problem of constructing matrix factorizations on cubic threefolds with an assigned shape.

\subsection{Betti tables and auxiliary modules}
\label{bettiCandidates}

Let us fix a cubic form $f \in S$. A matrix factorization of $f$ with shape (\ref{shape1412}) might be hard to construct. Nonetheless, the Shamash construction gives us a way to partially predict the shape of a matrix factorization arising as the 2-periodic part of the resolution of an arbitrary $S/f$-module $N$, provided that we know the Betti numbers $\beta_{i,j}(N)$ of $N$ as an $S$-module. Thus, a possible approach is to construct auxiliary $S$-modules $N$ giving rise over $S/f$ to a matrix factorization of $f$ with the desired shape.

For such $N$, how should its Betti table $\beta_{i,j}(N)$ look like? If we assume that no cancellation will occur when taking the minimal part of the Shamash resolution, i.e., the Shamash resolution is already minimal, a prescribed shape imposes linear conditions on the entries of a table $\beta_{i,j}$ filled with natural numbers. For instance, if we assume $\pd N < 5$, for the shape (\ref{shape1412}) such a table has the following form, up to twist:
\[
\bettif{
	0& 	\beta_{0,0}  &		\beta_{1,1}  & 	 . &	.	  & 	. \\
	1& 	\beta_{0,1}  &		\beta_{1,2}  & 	 \beta_{2,3} &	\beta_{3,4}  & 	. \\
	2& 	.  &		.  & 	 \beta_{2,4} &	\beta_{3,5}	  & 	\beta_{4,6} \\
	3& 	.  &		.  &   .&	.	& 	\beta_{4,7} 
}
\quad
\mbox{s.t.}
\quad
\left\{
\begin{array}{l}
\beta_{0,0}+\beta_{2,3}+\beta_{4,6} = 15\\ 
\beta_{0,1}+\beta_{2,4}+\beta_{4,7} = 2\\ 
\beta_{1,1}+\beta_{3,4}=2\\
\beta_{1,2}+\beta_{3,5}=15
\end{array}
\right.
\]
It turns out that a finite number of candidate Betti tables are allowed. As the transposed of a matrix factorization is again a matrix factorization, we could as well consider Betti tables giving rise to matrix factorizations with the dual shape
\[
\begin{array}{cc}
2 & .\\
15 & 15 \\
. & 2 \\
\end{array}
\]
We might also tolerate cancellations, i.e., we might assume that the Shamash resolution is not minimal; this makes the number of candidate Betti tables become infinite. However, we can always limit our search to finitely many cases, fixing for instance the entries of the tables in which we allow cancellations and an upper bound for their number.

By doing this, we end up with a list of tables; we can further limit our search to the ones lying in the Boij--S\"oderberg cone, i.e., tables $\beta_{i,j}$ for which there exists a rational number $q \in \mathbb{Q}$ and an $S$-module $M'$ such that $q\cdot\beta_{i,j}=\beta_{i,j}(M')$. It is of course convenient to let a computer deal with all the possibilities. 

\begin{example}
	\label{exampleList}
	A list of tables satisfying the aforementioned conditions can be produced by a \Mac computation, whose implementation is provided by the function \texttt{verifyAssertionsOfThePaper(3)} in \cite{SchreyerTanturriCode}. An example of a table in this list is
	\begin{equation}
	\label{auxiliaryBetti}
	\bettif{
		0& 	1 &		.  & 	  &		  & 	 \\
		1& 	  &		.  & 	  &		  & 	 \\
		2& 	  &		5 & 	  &		  & 	 \\
		3& 	  &		2 & 	15&		11& 	2 
	}	
	\end{equation}
	\end{example}

Suppose there exists an auxiliary $S$-module $N$ with resolution $F_{\bullet}$ with Betti numbers (\ref{auxiliaryBetti}), and consider a cubic form $f$. If we apply the Shamash construction to get a resolution of $N$, it is easy to see that the induced map $\overline{F_0}(-3) \rightarrow \overline{F_1}$ has a non-zero invertible part, hence the expected shape of the induced matrix factorization is (\ref{shape1412}).

The following proposition shows that such an auxiliary module $N$ exists and its induced matrix factorization has indeed the expected shape.

\begin{proposition}
\label{auxiliaryE}
Let $E$ be a general curve of genus 10 and degree 13 in $\PP^4$ and $X=\Vi(f)$ a general cubic threefold containing it. Then the Betti table of $S/I_E$ is (\ref{auxiliaryBetti}), the matrix factorization induced by $S/I_E $ on $X$ has shape (\ref{shape1412}) and is general enough in the sense of Theorem \ref{constructionThm}, i.e., it can be used to construct curves of genus 12 and degree 14.
\end{proposition}
\begin{proof}
	For such a curve $E$, all the statements correspond to open conditions and it is sufficient to check them on a particular example. An implementation of its construction is provided by the function \texttt{verifyAssertionsOfThePaper(4)} in \cite{SchreyerTanturriCode}; an explanation of the procedure used is to be found in the proof of Theorem \ref{unirationalityThm} and in Algorithm \ref{algorithmUnirat}.
	\end{proof}

\subsection{Unirationality of \texorpdfstring{$\sH_{12,8}$}{H 12 8}}

Summarizing, we can use general curves $E$ of genus 10 and degree 13 to get curves $C$ of genus 12 and degree 14. Moreover, such construction is unirational; this means that a unirational family of $E$'s yields a unirational family of $C$'s. Thus, we can focus on the former in the attempt of constructing a family dominating the latter.

\begin{theorem}
	\label{unirationalityThm}
The spaces $\sW^4_{12,14}$ and $\sH_{12,8}$ are unirational.
\end{theorem}
\begin{proof}
Let $H_{13,10} \subset \Hilb_{13t+1-10}(\PP^4)$ and  $H_{14,12} \subset \Hilb_{14t+1-12}(\PP^4)$ denote the components whose general elements are linearly normal non-degenerate smooth curves of degree and genus $({d},g)=(13,10)$ or $(14,12)$ respectively.
These components dominate $\sW^4_{10,13}$ and $\sW^4_{12,14}$. 

We will exhibit a unirational family of curves $C$ in $H_{14,12}$ by explicitly constructing a dominant family of curves $E$.
To do that, suppose we have a unirational parameterization of $\sM_{10,5}$, the moduli space of curves of genus 10 with 5 marked points; start from a curve $E$ and an effective divisor $D$ of degree 5. The linear system $|K_E-D|$ embeds $E$ in a curve of degree 13 in $\PP^4$ by Riemann--Roch. The construction dominates $H_{13,10}$  and via matrix factorizations this unirational parameterization induces a unirational family in $H_{14,12}$.

A unirational parameterization of $\sM_{10,5}$ can be constructed as follows. In \cite{GeissUnirationality}, a dominant unirational family of $6$-gonal curves $E$ of genus $10$ is constructed by means of liaison of curves in $\PP^1 \times \PP^2$. We can moreover modify the last step of the construction (see Algorithm \ref{algorithmUnirat} below) to impose $E$ to pass through five unirationally chosen points.

Thus we have produced a unirational family of curves in $H_{14,12}$, whose general element is a smooth irreducible curve of maximal rank with expected Betti table as in Proposition \ref{expected1412}. The corresponding code is implemented in the function \texttt{randomCurveGenus12Degree14InP4} of \cite{SchreyerTanturriCode}, along the lines of Algorithm \ref{algorithmUnirat}. It remains to prove that the family of curves constructed from  pairs $(E,X)$ with $E \in H_{13,10}$ and $X \in \PP(\HH^0(\sI_E(3)))$ via matrix factorizations dominates $H_{14,12}$. For this it suffices to prove that
we can recover $E$ from a matrix factorization $(\varphi,\psi)$ of shape (\ref{shape1412}).

\begin{proposition}
	\label{reconstruction1310}
	Let $E \in H_{13,10} $ be a general curve of genus $10$ and degree $13$, 
	let $X$ be a general cubic containing $E$ and let $\sF$ be the rank $7$ vector bundle on $X$ associated to the matrix factorization induced by $N=S/I_E$, i.e., $\sF$ is the image of $\psi$
	$$
	\xymatrix{  \sO^{15}_X(-3)   \oplus \sO^2_X(-4)& &\ar[ll]_{ \begin{pmatrix} \psi_{11} & \psi_{12} \cr 0 & \psi_{22} \cr    \end{pmatrix} }  \sO^2_X(-4) \oplus \sO^{15}_X(-5)}.
	$$
	There exists an exact complex induced by the Shamash construction
	$$
	0 \leftarrow \sI_{E/X} \leftarrow \sO^4_X(-3) \oplus \sO^2_X(-4) \leftarrow \sF \leftarrow \sO^2_X(-4) \leftarrow 0;
	$$
	moreover, for a general choice of a quotient $\sO_X^4(-3) \leftarrow \sO^{15}_X(-3)$ which composes to zero with the component
	$\psi_{1,1}$ of $\psi$, the complex 
	\begin{equation}
	\label{resolutionE'}
	\sO^4_X(-3) \oplus \sO^2_X(-4) \leftarrow \sF \leftarrow \sO^2_X(-4) \leftarrow 0
	\end{equation}
	is a locally free resolution of the ideal sheaf of a smooth curve $E'\in H_{13,10}$ on $X$.
	
	Let $(\psi,\varphi)$ be a given general matrix factorization on $X$ of shape (\ref{shape1412}) and let $\sF$ be the image of $\psi$. Then the choice of the quotient $q$ as above corresponds to the choice of a point in $\PP^4$; for a general such choice, (\ref{resolutionE'}) is a locally free resolution of the ideal sheaf of a smooth curve $E'\in H_{13,10}$ on $X$.
	\end{proposition}

\begin{proof} The first step is just reversing the Shamash construction of the $S_X$-resolution of $N=S_E$. 
	
Since $X$ is smooth the kernel of the map $ \sI_{E/X} \leftarrow \sO^4_X(-3) \oplus \sO^2_X(-4)$ is already a vector bundle $\sG$ on $X$. The bundle $\sF$ surjects onto $\sG$  with the image of $\sF \leftarrow \sO^2_X(-4)$ contained in the kernel. Since the kernel of the map
$\sG \leftarrow \sF$ is a rank 2 vector bundle of the same degree as $\sO^2_X(-4)$, the induced map between the kernel and $\sO^2_X(-4)$ is an isomorphism.

The fact that, for a given (general) matrix factorization, a general choice of the quotient $q$ yields a complex (\ref{resolutionE'}) which is a locally free resolution of a smooth curve $E'\in H_{13,10}$ is an open condition both on matrix factorizations and in $\PP^4$. It is thus sufficient to check it computationally on an explicit example, as can be done with the code provided by the function \texttt{verifyAssertionsOfThePaper(5)} in \cite{SchreyerTanturriCode}.
\end{proof}

Finally, to conclude with the unirationality of $\sH_{12,8}$ we note that a general point in $\sW^4_{12,14}$ gives as Serre dual model a point in $\sW^1_{12,8}$ and conversely. Moreover, the choice of a basis of $\PP^1$ is rational, and thus we get a unirational family of $\PP^1$-coverings of degree $8$. The locus of curves in $\sH_{12,8}$ having a smooth component of the Brill--Noether locus of expected dimension is open and contains the points we explicitly construct, hence our family is dominant. This completes the proof of Theorem  \ref{unirationalityThm}.

The function \texttt{randomGenus12Degree8CoverOfP1} in \cite{SchreyerTanturriCode} is an implementation of the above unirational construction and produces a random canonical curve of genus 12 together with two hyperplanes in $\PP^{11}$ cutting out a $g^1_8$. \qedhere
\end{proof}

\begin{remark}
Let 
$M^{15 \; 2}_ {\;2 \; 15}(X)$
 denote the  component, in the space of equivalence classes of   shape (\ref{shape1412}) on a given cubic $X$, whose general element is induced by a curve $C \in H_{14,12}$.  Above we have established a unirational correspondence between spaces of curves on $X$
$$\xymatrix{ \{C \subset X \} \ar[rd]_{\GG(2,5)} && \ar[dl]^{\PP^4} \{ E \subset X \} \cr
& M^{15 \; 2}_ {\;2 \; 15}(X) &\cr
}
$$
whose fibers are open subsets of a $\GG(2,5)$ or $\PP^4$ respectively.
We may interchange the role of $C$ and $E$: since $S_C$ and $\Gamma_*(\sO_E)$ have Betti tables 
$$
\bettif{0& 	1 &		.  & 	  &		  & 	 \\
			1& 	  &		.  & 	  &		  & 	 \\
			2& 	  &		4 & 	  &		  & 	 \\
			3& 	  &		5 & 	18&		12& 	2 

} \quad \hbox{ and } \quad
\bettit{
0& 1 & & & \cr
1 & . & & & \cr
2& 2 & 15 & 18 & 5 \cr
3 & & & & 1\cr
}
$$
they both lead to matrix factorizations on $X$ of shape  $$\begin{matrix} 15 & 2 \cr 5 & 18\end{matrix}\, $$ By the Reconstruction Theorem \ref{reconstructionThm}, and the same argument as in Proposition \ref{reconstruction1310}, we get another correspondence
$$\xymatrix{ \{C \subset X \} \ar[rd]_{\GG(2,5)} && \ar[dl]^{\PP^4} \{ E \subset X \} \cr
& M^{15 \; 2}_ {\; 5 \; 18}(X) &\cr
}.
$$
We believe that this symmetry can be explained by the fact that curves $C \in H_{14,12}$ are linked to curves $E \in H_{13,10}$ via a complete intersection of three cubics:
$$ \deg C + \deg E = 27=3^3 \hbox{ and } g_C-g_E =\frac{1}{2}(C-E).((9-5)H)=2.$$
This fact yields a correspondence
$$\xymatrix{ 
& \ar[dl]_{\PP^3} \{ \hbox{c.i.}\,  C \cup E \} \ar[dr]^{\GG(3,5)}& \cr
H_{14,12}&&  H_{13,10} \cr
}
$$
and a simpler proof that $H_{14,12}$ is unirational, as further shown in \cite[Remark 3.3]{KeneshlouTanturri}.
\end{remark}

\begin{algorithm}
\label{algorithmUnirat}
Summarizing, the following construction yields a unirational parameterization of $\sW^4_{12,14}$. The first four steps are a slight modification of the construction in \cite{GeissUnirationality}. The algorithm is implemented by the function \texttt{randomCurveGenus12Degree14InP4} in \cite{SchreyerTanturriCode}.
\begin{enumerate}
	\item On $\PP^1 \times \PP^2$, start with a rational curve of degree 4 together with 3 general lines. Call $E''$ their union.
	\item Choose two general forms $g_i \in \HH^0(\sI_{E''}(4,2))$ and construct $E'$ as the linkage of $E''$ on the complete intersection defined by $g_1, g_2$.
	\item Choose unirationally five general points $\{p_j\}$ in $\PP^1 \times \PP^2$ and choose, in the $7$-dimensional space $\HH^0(\sI_{E''}(3,3))$, two general forms $f_i$ vanishing on each $p_j$.
	\item Construct $E$ as the linkage of $E'$ in the complete intersection defined by $f_1, f_2$. By construction, $E$ passes through $p_j$, is a general curve of genus $10$ and  $D=p_1+\ldots+p_5$ is a general effective divisor of degree $5$ on $E$.
	\item Embed $E$ via $|K_E-D|$ into $\PP^4$. The curve $E \subset \PP^4$ is a general curve of genus $10$ and degree $13$.
	\item Choose a general cubic hypersurface $X \supset E$ and consider the matrix factorization on $X$ induced by $S/I_E$.
	\item Choose a general point $p \in \GG(2,5)$ as in Remark \ref{grassmannians}, construct the monad (\ref{monad}) and the corresponding curve $C \subset X$, which is a curve of genus 12 and degree 14.
\end{enumerate}
\end{algorithm}

\section{Families of curves on rational surfaces}
\label{families}	
	
In this section, we show how matrix factorizations can be used to construct unirational families of curves of genus $g$ and degree ${d}$ in $\PP^4$, with $(g,{d})$ belonging to
\[
\{(16,17), (17,18), (18,19), (19,20), (20,20)\}.
\]

The main motivation for the choice of these cases is the unknown unirationality of the corresponding moduli spaces of curves. One would like to produce a unirational family of projective curves which is dominant on the underlying moduli space of curves. As a general expectation, curves with fixed genus and lower degree should be easier to construct; the degree ${d}$ considered for each $g$ above is chosen as the minimum such that the Brill--Noether number $\rho(g,4,{d}) \geq 0$. 

\subsection{Explicit construction}

 We can try to mimic the technique used in Section \ref{bettiCandidates} and look for auxiliary modules whose Betti tables satisfy certain conditions.  
 A list of candidate Betti tables can be produced with the same technique and implementation used in Example \ref{exampleList}. Alternatively, the function \texttt{precompiledListOfCandidates} in \cite{SchreyerTanturriCode} prints precomputed lists for each genus $g \in [16,20]$.
 
 For instance, the lists contain the tables reported in Table \ref{bettiauxiliary}. All of them correspond to modules $N$ supported on a curve which will be denoted by $Z$. We will assume that $\sL=\widetilde N$ is a line bundle on $Z$. 
 	{\small
		\begin{table}[h!bt]
			\caption{Betti tables for auxiliary modules}
			\label{bettiauxiliary}
			\begin{tabular}{ccc} \toprule
				$(g,{d})$	 &	$\beta_{i,j}(N)$  & $(\codim \supp N, \deg N)$ \\ \midrule
				\rule{1.2ex}{0ex}$(16,17)$\rule{1.2ex}{0ex} & \rule{1.2ex}{0ex}$\bettif{
					0& 	6 &		10  & 	3 &		  & 	 \\
					1& 	  &		3 & 	  &		  & 	 \\
					2& 	  &		1 & 	13&		9 & 	1
				}$\rule{1.2ex}{0ex} &
$ (3, 19)$
				\\ \midrule
				\rule{1.2ex}{0ex}$(17,18)$\rule{1.2ex}{0ex} & \rule{1.2ex}{0ex}$\bettif{
					0& 	6 &		10  & 	3 &		  & 	 \\
					1& 	  &		3 & 	  &		  & 	 \\
					2& 	  &		2 & 	16&	   12 & 	2
				}$\rule{1.2ex}{0ex} &
$(3,18)$
				\\ \midrule
				\rule{1.2ex}{0ex}$(18,19)$\rule{1.2ex}{0ex} & \rule{1.2ex}{0ex}$\bettif{
					0& 	6 &		10  & 	3 &		  & 	 \\
					1& 	  &		3 & 	  &		  & 	 \\
					2& 	  &		3 & 	19&		15 & 	3
				}$\rule{1.2ex}{0ex}  &
$(3,17)$
				\\ \midrule
				\rule{1.2ex}{0ex}$(19,20)$\rule{1.2ex}{0ex} & \rule{1.2ex}{0ex}$\bettif{
					0& 	6 &		10  & 	3 &		  & 	 \\
					1& 	  &		3 & 	  &		  & 	 \\
					2& 	  &		4 & 	22&	18	  & 	4
				}$\rule{1.2ex}{0ex} &
$(3,16)$
				\\ \midrule
				\rule{1.2ex}{0ex}$(20,20)$\rule{1.2ex}{0ex} & \rule{1.2ex}{0ex}$\bettif{
					0& 	6 &		10  & 	3 &		  & 	 \\
					1& 	  &		4 & 	  &		  & 	 \\
					2& 	  &		 &	 	16&		 14& 	3
				}$\rule{1.2ex}{0ex} &
$(3,16)$
									\\  \bottomrule
			\end{tabular}
		\end{table}
	}

  The first row in these Betti tables is independent of $(g, {d})$ and the corresponding complex over $S$, dualized and sheafified,
  \begin{equation}
	\label{monadY}
	\xymatrix{
		0 \ar[r] &
		\sO_{\PP^4}^6(-4) \ar[r] &
		\sO_{\PP^4}^{10}(-3) \ar^-{\alpha}[r] &  
		\sO_{\PP^4}^{3}(-2) \ar[r] &
		0
	}
	\end{equation}
 could be a monad for the ideal sheaf of a surface $Y\subset \PP^4$.
 Two families of smooth surfaces of this kind are known: 
 \begin{itemize} \item the Alexander surfaces $Y$ \cite{Alexander}, $\PP^2$ blown up in $10$ general points embedded
 via the linear system $|14L-\sum_{i=1}^{10} E_i|$, where $L$ is the strict transform of a general line in $\PP^2$ and $E_i$ are the exceptional divisors corresponding to the 10 blown-up points, and 
 \item the blow-ups $Y'$ of  Enriques surfaces in a single point embedded by $|H-E|$, where $H$ is a Fano polarization and $E$ the exceptional divisor \cite{AureRanestad}. 
 \end{itemize}
 Both surfaces have degree $9$, $K_Y^2=-1$, sectional genus $\pi=6$ and as Hartshorne--Rao module $\HH^1_*(\sI_Y) =\coker(S^{10}(-3)\to S^3(-2))$ a module with Hilbert series $3t^2+5t^3+t^4$. They differ by the Betti numbers of their Hartshorne--Rao modules,
 which are
 $$
 \begin{tabular}{c|cccccc} 
  & 0 & 1 & 2 & 3 & 4 & 5 \cr\hline
 2 & 3 & 10 & 6 \cr
 3 &    &      & 15 &26 & 15 & 3 \cr
  4 &    &      &  1& 3&  3& 1\cr
   \end{tabular} 
\qquad  \hbox{ and } \qquad
\begin{tabular}{c|cccccc}
 & 0 & 1 & 2 & 3 & 4 & 5 \cr\hline
 2 & 3 & 10 & 6 \cr
 3 &    &      & 15 &25 & 12 & \cr
  4 &    &      &  & &  & 1\cr
  \end{tabular} 
 $$
respectively. Hence also $S_Y$ and $S_{Y'}$ have different Betti tables:
 $$
 \begin{tabular}{c|cccccc}
   & 0 & 1 & 2 & 3 & 4  \cr\hline
 0 & 1 & .\cr
 1& & .\cr
  2 & & .\cr
 3 & & .  \cr
 4 &        & 15 &26 & 15 & 3 \cr
  5 &          &  1& 3&  3& 1\cr
   \end{tabular} 
\qquad  \hbox{ and } \qquad
\begin{tabular}{c|cccccc}
  & 0 & 1 & 2 & 3 & 4  \cr\hline
 0 & 1 & . &\cr
 1& & .\cr
  2 & & .\cr
 3 & & .\cr
 4&         & 15 &25 & 12 & \cr
  5 &          &  & &  & 1\cr
   \end{tabular}
 $$
The rational surface $Y$ has a 6-secant line and contains no $(-1)$-line, while the Enriques surface has no 6-secant line and contains one $(-1)$-line.
For further details, see \cite{DeckerEinSchreyer}.

\begin{proposition}\label{intersectionXY}
If $C$ is a curve  of genus $g$ and degree ${d}$ obtained via matrix factorizations from an auxiliary module $N$ with Betti table as in Table \ref{bettiauxiliary} such that 
\begin{enumerate}
\item $\sL= \widetilde N$ is a line bundle on a curve ${Z}$ different from $C$, and
\item (\ref{monadY}) is a monad for a smooth surface $Y$ of degree $9$ as above, 
\end{enumerate}
then $C$ lies on $Y$. More precisely,
if $f \in (I_C)_4$ is any quartic which annihilates $N$ and $X=\Vi(f)$ the corresponding hypersurface, then
$$Y \cap X=C \cup {Z}.$$ 
\begin{proof} 
Since $Y$ does not lie on any quartic, the intersection $Y\cap X$ is proper
and the sequence (\ref{monadY}) restricted to $X$ 
	\begin{equation}
	\label{monadYX}
	\xymatrix{
		0 \ar[r] &
		\sO_{X}^6(-4) \ar[r] &
		\sO_{X}^{10}(-3) \ar[r] &  
		\sO_{X}^{3}(-2) \ar[r] &
		0
	}
	\end{equation}
is a monad for the ideal sheaf $\sI_{Y\cap X/X}$ of $Y\cap X$ on $X$. We claim that (\ref{monadYX}) is a subcomplex of the sheafified dual of the suitably twisted linear strand in the Shamash resolution of $N$.

For example, let us focus on the case $(g,{d})=(16,17)$. The dual linear strand reads
$$		0 \to
		\sO_{X}^{0+1}(-5) \to
		\sO_{X}^{6+13}(-4) \to
		\sO_{X}^{10+9}(-3) \to  
		\sO_{X}^{3+1}(-2) \to		0
$$
and the maps from a first to a second summand are all zero by (\ref{shamashzero}). Thus, we get a commutative diagram of monads
$$	\xymatrix{
		0 \ar[r] & \sO_{X}^6(-4) \ar[r] \ar[d] & \sO_{X}^{10}(-3) \ar[r] \ar[d] &  \sO_{X}^{3}(-2) \ar[r] \ar[d]& 0 \cr
		0 \ar[r] & \sO_{X}^3(-2)\oplus \sO_X(-1) \ar[r]  & \sF \ar[r]  &  \sO_{X}^{4}(-2) \ar[r] & 0 \cr
	}$$
where the first vertical map is up to sign a component of the dual of the first map of the $S_X$-resolution of $N$, and the third one is the inclusion induced by the Shamash resolution of $N$. The map on homology gives us a map $\sI_{Y\cap X/X} \to \sI_{C/X}$
between torsion free sheaves,
whose double dual is a map $\sO_X \to \sO_X$. Thus, to conclude that $C$ is a component of $Y \cap X$, it suffices to prove that $\sI_{Y\cap X/X} \to \sI_{C/X}$ is not the zero map. Let $\sJ$ and $\sK$ denote the kernels in the monads. We get a diagram
	$$
	\xymatrix{
		0 \ar[r] & \sO_{X}^6(-4) \ar[r] \ar[d] & \sJ \ar[r] \ar[d] &  \sI_{Y\cap X/X} \ar[r] \ar[d]& 0 \cr
		0 \ar[r] & \sO_{X}^3(-2)\oplus \sO_X(-1) \ar[r]  & \sK \ar[r]  &  \sI_{C/X} \ar[r] & 0 \cr
	}
$$
of exact sequences. 

If the map on the right was zero, we would get a homotopy $\sJ \to \sO_{X}^3(-2)\oplus \sO_X(-1)$, which since $\HH^1(\sO_X(n))=0$ for all $n$ would lift to a
map $\sO_X^{10}(-3) \to \sO_{X}^3(-2)\oplus \sO_X(-1)$ such that
$$
\xymatrix{
		 \sO_{X}^6(-4) \ar[r] \ar[d] & \sO_X^{10}(-3)  \ar[dl] \cr
		 \sO_{X}^3(-2)\oplus \sO_X(-1) & \cr
	}
$$
commutes. But this contradicts the fact that the map 
$$
\xymatrix{
S_X^6 & \ar[l] S_X^{10}(-1)\oplus S_X^3(-2) \oplus S_X(-3)
}
$$ 
is the first map in the minimal free resolution of $N$ as
an $S_X$-module. 

Therefore, $C$ is a component of $Y\cap X$. The curve ${Z}$ is also contained in $Y\cap X$. Since
$$\deg C + \deg {Z} = \deg C + \deg N =36=\deg Y \deg X$$
there are no further components, and $C \cup {Z} = Y\cap X$. The proof for the other pairs $(g,{d})$ is similar.
\end{proof}
\end{proposition}

	\subsection{Families of curves on rational surfaces}
	\label{familiesOfCurves}
	
We have two ways to tackle the construction of our curves $C$: we could try to produce a module $N$ having a Betti table as in Table \ref{bettiauxiliary}, then induce a matrix factorization and get a curve as described in the previous sections. A key observation is that the line bundle $\sL$ on the  curve $Z$ coincides with $\left.\omega_Y(1)\right|_Z$.
This approach works, and led us to discover Proposition \ref{intersectionXY} and the fact that some of desired curves $C$ lie on Alexander surfaces. An implementation of the construction of curves on Alexander surfaces via matrix factorizations is provided by the function \texttt{verifyAssertionsOfThePaper(6)} in \cite{SchreyerTanturriCode}.

A second, more convenient approach is to look for our desired curves $C$ directly on these surfaces, e.g., the Alexander surfaces $Y$. The genus and the degree of $C$ impose conditions on the divisor class $[C] =a_0L-\sum a_iE_i \in \Pic(Y)$.
By maximizing the dimension of the linear systems, we can maximize the dimension of the corresponding unirational families of curves. In Table \ref{dimunirationalfam} we list the linear systems achieving the maximal dimension; a general element in such linear systems is a curve which satisfies all our assertions, as one can verify by computing a single randomly chosen example, see the code provided by the function \texttt{verifyAssertionsOfThePaper(7)} in \cite{SchreyerTanturriCode}. In particular this proves the first two assertions of Theorem \ref{constructionThm}.

{\small
	\begin{table}[h!bt]
		\caption{Unirational families of curves on the Alexander surface}
		\label{dimunirationalfam}
		\begin{tabular}{ccc} \toprule
			$(g,d)$ &linear system & dimension\\\midrule
			\rule{1.2ex}{0ex}$(16,17)$\rule{1.2ex}{0ex} 
			& $21L-\sum_{i=1}^{4}7E_i-\sum_{j=5}^{10}6E_j$
			& \rule{1.2ex}{0ex} 26 \rule{1.2ex}{0ex} 
			\\ \midrule
			\rule{1.2ex}{0ex}$(17,18)$\rule{1.2ex}{0ex} 
			& $22L-\sum_{i=1}^{8}7E_i-6E_9-5E_{10}$
			& \rule{1.2ex}{0ex} 27 \rule{1.2ex}{0ex} 
			\\ \midrule
			\rule{1.2ex}{0ex}$(18,19)$\rule{1.2ex}{0ex} 
			& $19L-\sum_{i=1}^{7}6E_i-\sum_{j=8}^{10}5E_j$
			& \rule{1.2ex}{0ex} 29 \rule{1.2ex}{0ex} 
			\\ \midrule
			\rule{1.2ex}{0ex}$(19,20)$\rule{1.2ex}{0ex} 
			& $20L-7E_1-7E_2-\sum_{i=3}^{8}6E_i-5E_9-5E_{10}$
			& \rule{1.2ex}{0ex} 30 \rule{1.2ex}{0ex} 
			\\ \midrule
			\rule{1.2ex}{0ex}$(20,20)$\rule{1.2ex}{0ex} 
			& $20L-7E_1-\sum_{i=2}^{9}6E_i-5E_{10}$
			& \rule{1.2ex}{0ex} 31 \rule{1.2ex}{0ex} 
			\\  \bottomrule
		\end{tabular}
	\end{table}
}

Unfortunately, the so-constructed unirational families are far from being dominant on the corresponding moduli spaces.
Curves of same degree and genus on a blown-up Enriques surface give at best families of the same dimension.

There are many other possible choices of a candidate Betti table of $N$.
 For instance, for $g \geq 16$, other even simpler rational surfaces show up and we can 
 produce other examples of curves lying on them. Unfortunately, 
all the unirational families we have been able to construct are not dominant. Nonetheless, there is no reason why one should not be able to realize bigger families of projective models via matrix factorizations starting from different Betti tables, the biggest obstacle being of course the construction of suitable auxiliary modules $N$.

\subsection{Curves lying on a general hypersurface}
We conclude by showing that, even though the examples of curves of genus $g \geq 16$ are far from being general as projective models, we can still use them, as well as the examples of curves with lower genera constructed in the previous sections, to prove that a general hypersurface contains a whole family of them.
 
 \begin{theorem}
	\label{generalQuartic}
	A general cubic hypersurface in $\PP^4$ contains a family of dimension $2{d}$ of curves of genus $g$ and degree ${d}$ for
	\[
	(g,{d}) \in \{(12,14), (13,15)\}.
	\]
	A general quartic hypersurface in $\PP^4$ contains a ${d}$-dimensional family of curves of genus $g$ and degree ${d}$ for
	\[
	(g,{d}) \in \{(16,17), (17,18), (18,19), (19,20), (20,20)\}.
	\]	
\begin{lemma}
\label{eulerNormal}
Let $C$ be a curve of genus $g$ and degree ${d}$ in $\PP^n$ and $X$ a hypersurface of degree $s$ containing it. Then 
\[
\chi(\sN_{C/X})={d} (n+1-s)+(1-g)(n-4).
\]
\begin{proof}
	The Euler sequence of $\PP^n$ restricted to $C$ yields
	\[
	\chi(\left.\sT_{\PP^{n}}\right|_{C})=(n+1)({d}+1-g)-1+g.
	\]
	Since $\left.\sN_{X/\PP^{n}}\right|_{C} \cong \sO_{C}(s)$, from the sequence defining $\sN_{X/\PP^{n}}$ restricted to $C$ we get
	\[
	\chi(\left.\sT_{X}\right|_{C}) = (n+1)({d}+1-g)-1+g - ({d} s+1-g).
	\]
	The conclusion follows by looking at the short exact sequence defining $\sN_{C/X}$.	
\end{proof}
\end{lemma}	
	
\begin{proof}[Proof of Theorem \ref{generalQuartic}]
	Let $C$ be a general curve in $\PP^4$ of genus $g$ and degree ${d}$, and let $X$ be a general hypersurface of degree $s$ containing it, with $s$ chosen accordingly to $(g,{d})$ as in the statement of the Theorem. By Lemma \ref{eulerNormal}, $\chi(\sN_{C/X})={d} (5-s)$.
	
	We claim that $\hh^{1}(\sN_{C/X})=0$. It is sufficient to check this vanishing on one example for each pair $(g,{d})$, as can be done with the \Mac code provided by the function \texttt{verifyAssertionsOfThePaper(8)} in \cite{SchreyerTanturriCode}, and conclude by semicontinuity. Hence, $\hh^0(\sN_{C/X})={d}(5-s)$.
	
	Let $\sT_s$ be the space of threefolds of degree $s$ containing a general curve $C$ of genus $g$ and degree ${d}$, up to projective equivalences. Let $m:=\hh^0(\PP^4,\sI_C(s))-1=\binom{4+s}{4}-s{d}+g-2$. We have
		\[
		\dim(\sT_s)=\dim \sM_g + \rho({d},4,g) + m - \hh^0(\sN_{C/X}) = \binom{4+s}{4}-25. \qedhere
		\]
	\end{proof}
\end{theorem}


\makeatletter
  \providecommand\@dotsep{5}
\makeatother
 \listoftodos\relax	
	
\end{document}